\documentclass{amsart}
\usepackage{verbatim}
\usepackage{rotating} 
\usepackage{amssymb}
\usepackage{mathrsfs,amsfonts,amsmath,amssymb,epsfig,amscd,xy,amsthm,pb-diagram} 

\newtheorem{theorem}{Theorem}[section]
\newtheorem{proposition}[theorem]{Proposition}
\newtheorem{notation}[theorem]{Notation}
\newtheorem{lemma}[theorem]{Lemma}

\newtheorem{corollary}[theorem]{Corollary}
\newtheorem{question}[theorem]{Question}
\newtheorem{definition}[theorem]{Definition}

\theoremstyle{plain}

\theoremstyle{remark}

\newtheorem{remark}[theorem]{Remark}
\newtheorem{example}[theorem]{Example}

\newtheorem{assumption}[theorem]{Assumption}

\newcommand{\C}{{\mathbb C}}

\newcommand{\Q}{{\mathbb Q}}

\newcommand{\R}{{\mathbb R}}
\newcommand{\Z}{{\mathbb Z}}
\newcommand{\N}{{\mathbb N}}

\newcommand{\Qbar}{\bar{\Q}}

\DeclareMathOperator{\Norm}{N}

\newcommand{\bP}{{\mathbb P}}

\newcommand{\bfx}{{\mathbf x}}
\newcommand{\bfv}{{\mathbf v}}

\newcommand{\bfu}{{\mathbf u}}

\newcommand{\scrN}{\mathscr{N}}

\newcommand{\scrQ}{\mathscr{Q}}

\author{Khoa D.~Nguyen}
\address{
Khoa D.~Nguyen \\
Department of Mathematics and Statistics\\
University of Calgary\\
AB T2N 1N4, Canada
}
\email{dangkhoa.nguyen@ucalgary.ca}

\keywords{Algebraic numbers, transcendental numbers, Subspace Theorem}
\subjclass[2010]{Primary: 11J87. Secondary: 11B39.}

\begin{document}
	\title[Series of reciprocals of Fibonacci and Lucas numbers]{Transcendental series of reciprocals of Fibonacci and Lucas numbers}
	
	\date{September 2020}
	
	\begin{abstract}
	Let $F_1=1,F_2=1,\ldots$ be the Fibonacci sequence. Motivated by
	the identity
	$\displaystyle\sum_{k=0}^{\infty}\frac{1}{F_{2^k}}=\frac{7-\sqrt{5}}{2}$, Erd{\"o}s and Graham asked whether $\displaystyle\sum_{k=1}^{\infty}\frac{1}{F_{n_k}}$ is irrational for
	any sequence of positive integers $n_1,n_2,\ldots$
	with $\frac{n_{k+1}}{n_k}\geq c>1$. We resolve the transcendence counterpart of their question: as a special case of our main theorem, we have that $\displaystyle\sum_{k=1}^{\infty}\frac{1}{F_{n_k}}$ is transcendental when
	$\frac{n_{k+1}}{n_k}\geq c>2$. 
	The bound $c>2$ is best possible thanks to the identity at the beginning. This paper provides a new way to apply the Subspace Theorem to obtain transcendence results and extends previous non-trivial results obtainable by only Mahler's method for special sequences of the form $n_k=d^k+r$.
	\end{abstract}
	
	\maketitle
	
	\section{introduction}\label{sec:intro}
	Let $F_1=1,F_2=1,\ldots$  be the Fibonacci sequence and let
	$L_1=1$, $L_n=F_{n-1}+F_n$ for $n\geq 2$ be the Lucas sequence. 
	In the 
	chapter   ``Irrationality and Transcendence'' of their 
	book \cite[p.~64--65]{EG80_OA}, starting from
	the Millin series:
	$$\sum_{k=0}^{\infty}\frac{1}{F_{2^k}}=\frac{7-\sqrt{5}}{2},$$
	Erd\"os and Graham asked the following:
	\begin{question}[Erd\"os-Graham, 1980]\label{q:EG}
	Is it true that $\displaystyle\sum_{k=1}^{\infty}\frac{1}{F_{n_k}}$ is irrational for any sequence $n_1<n_2<\ldots$ with $\frac{n_{k+1}}{n_k}\geq c>1$?
	\end{question}

	The transcendence counterparts of this and many questions in 
	\cite[Chapter~7]{EG80_OA} were implicit throughout 
	 the chapter, hence its title. Indeed, this 
	topic inspired intense research activities 
	most of which involved the so called Mahler's method.
	As a consequence of our main result (see Theorem~\ref{thm:general}), we resolve the 
	transcendence version of Question~\ref{q:EG} even when one is allowed to \emph{randomly mix} the Fibonacci and Lucas numbers:
	\begin{theorem}\label{thm:Fibonacci and Lucas}
	Let $c>2$ and let $n_1<n_2<\ldots$ be positive integers
	such that $\displaystyle\frac{n_{k+1}}{n_k}\geq c$ for every $k$. 
	Then the number $\displaystyle\sum_{k=1}^\infty\frac{1}{f_k}$
	is transcendental where $f_k\in\{F_{n_k},L_{n_k}\}$ for every $k$.
	\end{theorem}

	Although Fibonacci and Lucas numbers 
	have been discovered for hundreds of years, some of
	their basic properties 
	have been established only recently thanks to
	powerful modern methods, for example \cite{BMS06_CA,Ste13_OD}.
	The key ingredient of the proof of our main theorems is a new 
	application of the Subspace Theorem in treating transcendence 
	of series in which it is hard to control the denominators 
	of the partial sums. Before providing more details about
	the method, let us provide a very brief and incomplete survey of known
	results on irrationality and transcendence of sums like 
	$\displaystyle\sum_{k=1}^{\infty}\frac{1}{F_{n_k}}$. 
	In a nutshell,
	we can divide previous results in two groups. 
	
	\medskip
	
	The first group treats series in which the  denominators 
	of the partial sums are very small compared to the reciprocals of 
	the error terms.
	This includes work of Mignotte on the transcendence of
	$\displaystyle\sum_{k=1}^{\infty}\frac{1}{k!F_{2^k}}$ and 
	$\displaystyle\sum_{k=1}^{\infty}\frac{2+(-1)^k}{F_{2^k}}$ \cite{Mig71_QP,Mig77_AA}. 
	In fact, these results predate Erd\"os-Graham question. 
	As another example, consider $s:=\displaystyle\sum_{k=0}^{\infty}\frac{1}{F_{2^k+1}}$. For every sufficiently large integer $N$, thanks to  divisibility properties of the Fibonacci sequence and the fact that $2^k+1$ divides $2^{3k}+1$, the denominator of 
		$\displaystyle\sum_{k=1}^{N}\frac{1}{F_{2^k+1}}$ is at most
		$\displaystyle\prod_{k=\lfloor N/3\rfloor}^N F_{2^k+1}$
		which is $o(F_{2^{N+1}+1})$, hence $s$ must be irrational. We refer the readers to \cite{Bad93_AT} and the references there 
		for similar results. In fact, if $n_1<n_2<\ldots$
		satisfies $\displaystyle\frac{n_{k+1}}{n_k}\geq c>2$ then 
		$\displaystyle\sum_{k=1}^{\infty}\frac{1}{F_{n_k}}$ is irrational since $F_{n_1}\cdots F_{n_N}=o(F_{n_{N+1}})$ as $N\to\infty$. Likewise, if $\displaystyle\frac{n_{k+1}}{n_k}\geq c>3$ then $\displaystyle\sum_{k=1}^{\infty}\frac{1}{F_{n_k}}$ is transcendental
		by applying Roth's theorem and the fact that 
		$F_{n_1}\cdots F_{n_N}=O\left(F_{n_{N+1}}^{(1/2)-\epsilon}\right)$
		for an appropriate $\epsilon>0$. 
		However, replacing
		those \emph{easy} bounds 
		$2$ and $3$ respectively by smaller numbers for irrationality
		and transcendence problems appears to be 
		a \emph{very difficult} task.
		
		\medskip		
		
		The second group constitutes the majority of results in 
		this topic. 
	In 1975, Mahler \cite{Mah75_OT} reproved Mignotte's result 
	using the method he had invented nearly 50 years earlier (see 
	Nishioka's notes \cite{Nis96_MF} for an introduction
	to Mahler's method). This method is applicable when the sequence
	$n_k$ has the 
	special form $n_k=d^k+r$ where $d,r\in\Z$ with $d\geq 2$. We refer 
	the readers to \cite{BT94_TR,DKT02_TO,KKS09_TO} and the references
	there for further details.	
	In fact, before this paper, there has not been one result
	establishing the transcendence of
	$\displaystyle\sum_{k=1}^\infty\frac{1}{F_{n_k}}$
	for \emph{an arbitrary sequence} $n_k$ with $\displaystyle\frac{n_{k+1}}{n_k}\geq c$
	for any single value $c<3$ (as explained above, $c>3$ is the easy bound due to an immediate application of Roth's theorem).
	
	\medskip
	
	\emph{From now on}, let $\alpha\neq \pm 1$ be a \emph{real quadratic unit};
	this means $\Q(\alpha)$ is real quadratic and $\alpha\neq \pm 1$ 
	a unit in
	the ring of algebraic integers. Let $\sigma$ be the non-trivial automorphism of $\Q(\alpha)$ and let $\beta=\sigma(\alpha)$. Without loss of generality, we assume
	$\vert\beta\vert<1<\vert\alpha\vert$. Let $H$ and $h$ respectively 
	denote the 
	absolute multiplicative and 
	logarithmic Weil height on $\Qbar$, our main result is the following:
	\begin{theorem}\label{thm:general}
	Let $a_n,b_n,c_n$ for $n\geq 1$ be sequences of real numbers with the following properties:
	\begin{itemize}
		\item For every $n\geq 1$, $c_n\in \Q$, $a_n,b_n\in \Q(\alpha)$,
		and $u_n:=a_n\alpha^n-b_n\beta^n\in \Q$.
		
		\item $\displaystyle\lim_{n\to\infty}\frac{h(a_n)}{n}=\lim_{n\to\infty}\frac{h(b_n)}{n}=\lim_{n\to\infty}\frac{h(c_n)}{n}=0$. 
	\end{itemize}
	Let $c>2$ and let $n_1<n_2<\ldots$ be positive integers such that
	$\displaystyle\frac{n_{k+1}}{n_k}\geq c$, 
	$u_{n_k}\neq 0$, and $c_{n_k}\neq 0$ for every $k$. Then the series 
	$\displaystyle\sum_{k=1}^\infty\frac{c_{n_k}}{u_{n_k}}$
	is transcendental.
	\end{theorem}

	\begin{example}
	Let $n_1<n_2<\ldots$ be as in 
	Theorem~\ref{thm:Fibonacci and Lucas}. Let $s=\displaystyle\sum_{k=1}^\infty\frac{1}{f_k}$ where 
	$f_k\in \{F_{n_k},L_{n_k}\}$ for every $k$. Note that $F_n=\displaystyle\frac{\alpha^n-\beta^n}{\sqrt{5}}$ and
	$L_n=\alpha^n+\beta^n$ with $\alpha=\displaystyle\frac{1+\sqrt{5}}{2}$
	and $\beta=\displaystyle\frac{1-\sqrt{5}}{2}$. We define $c_n=1$ for every $n$.
	If $n\notin\{n_k:\ k\geq 1\}$, we define $a_n=b_n=0$.  If
	$n=n_k$ and $f_k=F_{n_k}$, define $a_n=b_n=\displaystyle\frac{1}{\sqrt{5}}$. Finally if $n=n_k$ and $f_k=L_{n_k}$,
	define $a_n=1$ and $b_n=-1$. This explains why Theorem~\ref{thm:Fibonacci and Lucas} is a special case of Theorem~\ref{thm:general}. 
	\end{example}

	\begin{example}
	One can consider linear recurrence sequences of rational numbers of
	the form $u_n=A(n)\alpha^n+B(n)\beta^n$ where $A(t),B(t)\in \Q(\alpha)[t]$. Then Theorem~\ref{thm:general} implies
	that $\displaystyle\sum\frac{1}{u_{n_k}}$ is transcendental
	since we may choose $c_n=1$, $a_n=A(n)$, and $b_n=-B(n)$. Note
	that $h(a_n)=O(\log n)$ and $h(b_n)=O(\log n)$ in this case.
	\end{example}
	
	There have been two general transcendence results using the 
	Subspace Theorem recently and both involve values of a power series
	$\sum d_nz^n$ at an algebraic number $z_0$. One is a result
	of Adamczewski-Bugeaud \cite{AB07_OT} extending an earlier work
	of Troi-Zannier \cite{TZ99_NO}. In their work, the coefficients
	of $d_n$'s form an automatic sequence 
	and $z_0$ is the reciprocal of a Pisot number. The authors rely
	on the repeating pattern of automatic sequences to 
	apply the Subspace 
	Theorem using linear forms in three variables. The other is
	a result of Corvaja-Zannier \cite{CZ02_SN} 
	treating the case that $\sum d_nz^n$ is lacunary with 
	positive real coefficients and $z_0\in (0,1)$. 
	The problem considered here is  
	different from all the above. While it is true that one can express
	$\displaystyle\sum_{k=1}^\infty\frac{1}{F_{n_k}}$ as the value
	of a power series at 
	$1/\alpha$ with $\alpha=\displaystyle\frac{1+\sqrt{5}}{2}$, 
	neither the coefficients are automatic nor the series is lacunary
	for an \emph{arbitrary} choice of $n_k$
	with $n_{k+1}/n_k\geq c>2$. 
	
	\medskip
	
	The more subtle difference and key 
	reason for the difficulty in settling our current problem are as 
	follows. Let us consider the example 
	$s=\displaystyle\sum_{k=1}^\infty\frac{1}{F_{n_k}}$
	and $\alpha=\displaystyle\frac{1+\sqrt{5}}{2}$ again. Let
	$s_N=\displaystyle\sum_{k=1}^N\frac{1}{F_{n_k}}$ 
	be the sequence of partial sums so that
	$\vert s-s_{N-1}\vert = O(\vert\alpha\vert^{-n_N})$. Now the usual
	idea is to fix a large integer $P$, then
	truncate each
	$$\frac{1}{F_{n_{N-P+i}}}=s_{N,i}+O(\vert\alpha\vert^{-n_N})$$ 
	where $s_{N,i}$ is a finite sum of units for $1\leq i\leq P-1$.
	Then we have:
	$$\vert s-s_{N-P}-s_{N,1}-\ldots-s_{N,P-1}\vert=O(\vert\alpha\vert^{-n_N})$$
	and after assuming that $s$ is algebraic, 
	one might attempt to apply the Subspace Theorem to this 
	equation for $s$, $s_{N-P}$ and the individual units 
	in each $s_{N,i}$. The difference compared to work of 
	Adamczewski-Bugeaud or Corvaja-Zannier is that while terms in
	their application of the Subspace Theorem are 
	$S$-integers 
	for an appropriate choice of a finite set of places $S$ in 
	an appropriate number field, here
	we \emph{cannot} find such an $S$ so that
	the term $s_{N-P}$ above is an $S$-integer for infinitely many $N$. 
	For this reason, in our situation, when applying
	the Subspace Theorem we may have the contribution
	$H(s_{N-M})^D$ where $D$ is the number of terms. With just the 
	constraint $c>2$, it is entirely possible for the above 
	contribution to 
	offset
	the error term $O(\vert\alpha\vert^{-n_N})$ and one fails to apply 
	the Subspace Theorem. Therefore new ideas are 
	needed to overcome this crucial issue. Moreover, after one applies the Subspace Theorem, it remains a highly nontrivial task to
	arrive at the desired conclusion from the resulting linear relation. This paper promotes the innovation that one should start with
	a certain ``minimal expression'' before applying the Subspace Theorem in order to maximize the benefit of the resulting linear relation. We refer the readers to
	 the discussion right 
	after
	Proposition~\ref{prop:error when truncating} for more details.

	\textbf{Acknowledgements.}
	The author wishes to thank Professors Yann Bugeaud and Maurice Mignotte for useful comments. 
	The author is 
	partially supported by an NSERC Discovery Grant and a CRC II 
	Research Stipend from the Government of Canada. 
	
	\section{The Subspace Theorem}	 	
  	The Subspace
  	Theorem is one of the milestones of diophantine geometry in
  	the last 50 years. The first version was obtained by 
  	Schmidt \cite{Sch70_SA} and further versions were obtained
  	by Schlickewei and Evertse \cite{Sch90_TQ,Eve96_AI,ES02_AQ}. This 
  	section follows the exposition in the book of 
  	Bombieri-Gubler \cite{BG06_HI}.

  	Let $M_{\Q}=M_{\Q}^{\infty}\cup M_{\Q}^0$ where
  	$M_{\Q}^0$ is the set of $p$-adic valuations
  	and $M_{\Q}^{\infty}$
  	is the singleton consisting of the usual archimedean 
  	valuation. More generally, for every number field $K$,
  	write $M_K=M_K^\infty\cup M_K^0$ where
  	$M_K^\infty$ is the set of archimedean places and 
  	$M_K^0$ is the set of finite places. 
  	Throughout this paper, we fix an embedding of $\Qbar$ into $\C$ and
  	let $\vert\cdot\vert$ denote the usual absolute value on $\C$.
  	Hence for a number field $K$, the set $M_K^\infty$ corresponds to
  	the set of real embeddings and pairs of complex-conjugate
  	embeddings of $K$ into $\C$.
  	For every
  	$w\in M_K$, let $K_w$ denote the completion of 
  	$K$ with respect to $w$ and denote
  	$d(w/v)=[K_w:\Q_v]$ where $v$ is the restriction of
  	$w$ to $\Q$. Following \cite[Chapter~1]{BG06_HI}, 
  	for every $w\in M_K$ restricting to $v$ on $\Q$,
  	we normalize $\vert \cdot\vert_w$ as follows:
  	$$\vert x\vert_w = \vert \Norm_{K_w/\Q_v}(x) \vert_v^{1/[K:\Q]}.$$
  	Let $m\in\N$, for every vector $\bfu=(u_0,\ldots,u_m)\in K^{m+1}\setminus\{\mathbf 0\}$ and $w\in M_K$,
  	let $\vert \bfu\vert_w:=\displaystyle\max_{0\leq i\leq m} \vert u_i\vert_w$.
  	For $P \in \bP^m(\Qbar)$, let $K$ be a number
  	field such that $P$ has a representative 
  	$\bfu\in K^{m+1}\setminus\{\mathbf 0\}$
  	and define:
  	$$H(P)=\prod_{w\in M_K} \vert \bfu\vert_w.$$
  	It is an easy fact that this is independent of the choice 
  	of $\bfu$ and the number field $K$. 
  	Then we define $h(P)=\log (H(P))$. For $\alpha\in \Qbar$, 
  	write $H(\alpha)=H([\alpha:1])$
  	and $h(\alpha)=\log(H(\alpha))$.  	
  	Later on, we will use the classical version of 
  	Roth's theorem \cite[Chapter~6]{BG06_HI} to
  	give a weak upper bound on $n_{k+1}/n_k$:
  	\begin{theorem}[Roth's theorem]\label{thm:Roth}
  	Let $\kappa>2$. Let $s$ be a real algebraic number. Then
  	there are only finitely many rational numbers
  	$s'$ such that 
  	$$\vert s'-s\vert\leq H(s')^{-\kappa}.$$
  	\end{theorem}
  	
  	Let $m\in\N$, for every vector $\bfx=(x_0,\ldots,x_m)\in K^{m+1}\setminus\{\mathbf 0\}$, let $\tilde{\bfx}$ denote
  	the corresponding point in $\bP^{m}(K)$. For every $w\in M_K$,
  	denote $\vert \bfx\vert_w:=\displaystyle\max_{0\leq i\leq m} \vert x_i\vert_w$. We have:
  	
 	\begin{theorem}[Subspace Theorem]\label{thm:Subspace}
 	Let $n\in\N$, let $K$ be a number field,
 	and let $S\subset M_K$ be finite. For every $v\in S$, let 
 	$L_{v0},\ldots,L_{vn}$ be linearly independent 
 	linear forms in the variables $X_0,\ldots,X_n$
 	with $K$-algebraic coefficients in $K_v$. For every 
 	$\epsilon>0$, the
 	solutions $\bfx\in K^{n+1}\setminus\{\mathbf{0}\}$
 	of the inequality:
 	$$\displaystyle\prod_{v\in S}\prod_{j=0}^n \displaystyle\frac{\vert L_{vj}(\bfx)\vert_v}{\vert \bfx\vert_v}\leq H(\tilde{\bfx})^{-n-1-\epsilon}$$
 	are contained in finitely many hyperplanes of
 	$K^{n+1}$.
 	\end{theorem}

	\section{Preliminary results and preparation for the proof
	of Theorem~\ref{thm:general}}\label{sec:prelim}
	Throughout this section, we assume the notation in the statement of
	Theorem~\ref{thm:general} and put
	$A_k=a_{n_k}$, $B_k=b_{n_k}$, $C_k=c_{n_k}$, $U_k=u_{n_k}$
	for every $k$ to simplify the notation.
	\emph{From now on}, assume that 
	$s:=\displaystyle\sum_{k=1}^\infty\frac{C_k}{U_k}$ 
	is algebraic and let
	$K$ be the Galois closure of $\Q(\alpha,s)$.
	For $m\geq 1$, let 
	$s_m=\displaystyle\sum_{k=1}^m\frac{C_k}{U_k}$ be
	the sequence of partial sums. We will repeatedly use the 
	following observation:
	if $(t_n)_n$ is a sequence in $K^*$ such that 
	$h(t_n)/n\rightarrow 0$ as $n\to\infty$ then
	$\vert t_n\vert_v=e^{o(n)}$ as $n\to\infty$ for every $v\in M_K$; this means for every $\epsilon>0$, we have
	$e^{-\epsilon n}<\vert t_n\vert_v<e^{\epsilon n}$ for all sufficiently large $n$.

	\medskip
	
	We are given that $C_k\neq 0$ for every $k$. We may assume that
	$A_kB_k\neq 0$ for every $k$, as follows. Let $\sigma$ denote
	the nontrivial automorphism of $\Q(\alpha)$. Suppose that 
	$A_k=0$ then $U_k=-B_k\beta^{n_k}\in \Q^*$. 
	Applying $\sigma$ gives
	$U_k=-B_k\beta^{n_k}=-\sigma(B_k)/\alpha^{n_k}$, therefore
	$\sigma(B_k)/B_k=(\alpha/\beta)^{n_k}$. 
	Since $\vert \sigma(B_k)/B_k\vert =e^{o(n_k)}$, we conclude that
	$A_k=0$ is possible for only finitely many $k$. A similar conclusion
	holds for $B_k=0$ too. By ignoring the first finitely many $n_k$'s,
	we may assume $A_kB_k\neq 0$ for every $k$. 
	
	\medskip
	
	Similarly, from $\displaystyle\sum_{k=m_1}^{m_2}\frac{C_k}{U_k}=\frac{C_{m_1}}{U_{m_1}}+O(C_{m_1+1}/U_{m_1+1})$ for
	any $m_1<m_2\leq\infty$, by ignoring the first finitely many $n_k$'s, we may assume:
	\begin{equation}\label{eq: the s and sk's}
		\text{The numbers $s$ and $s_k$'s for $k=1,2,\ldots$
		are pairwise distinct and non-zero.}
	\end{equation}
	
	We start with several easy estimates:	
	\begin{lemma}\label{lem:sum of indices bound}
	\begin{itemize}
	\item [(i)] For every positive integer $m$, we have:
	$$(c-1)(n_1+\ldots+n_m) < n_{m+1}.$$
	\item [(ii)] For any positive integers $m<N$, we have:
	$$c^{N-m-1}(c-1)(n_1+\ldots+n_m)<n_N.$$ 
	\end{itemize}
	\end{lemma}
	\begin{proof}
	Part (i) follows from 
	$$\frac{n_1+\ldots+n_m}{n_{m+1}}< \frac{1}{c}+\frac{1}{c^2}+\ldots.$$  
	Part (ii) follows from $n_N\geq c^{N-m-1}n_{m+1}$ and part (i).
	\end{proof}
	
	\begin{lemma}\label{lem:height bound}
	\begin{itemize}
		\item [(i)] $H(u_n)=\vert\alpha\vert^{n+o(n)}$ and $H(c_n/u_n)=\vert\alpha\vert^{n+o(n)}$ as $n\to\infty$.
		\item [(ii)] $H(s_N)\leq \vert\alpha\vert^{n_1+\ldots+n_N+o(n_N)}$ as $N\to\infty$.
	\end{itemize}
	\end{lemma}
	\begin{proof}
	Due to the fact that $u_n\in\Q$, $\alpha$ and $\beta$ are units, 
	and our assumption on the
	$A_n$'s and $B_n$'s, we have
	$\vert u_n\vert=\vert \alpha\vert^{n+o(n)}$
	while the non-archimedean contribution
	is $e^{o(n)}$. This proves the first assertion of part (i), the remaining one follows since $H(c_n)=\vert\alpha\vert^{o(n)}$ as $n\to\infty$. 
	
	For part (ii), we use the inequality:
	$$H(s_N)\leq N\prod_{i=1}^N H(c_{n_i}/u_{n_i}).$$ 
	There exists $\delta_1>0$ such that 
	$H(c_{n_i}/u_{n_i})\leq \vert\alpha\vert^{n_i+\delta_1n_i}$ for every $i$ by part (i).
	Given any $\epsilon>0$, part (i) also gives that  
	$H(u_{n_i})\leq \vert\alpha\vert^{n_i+\epsilon n_i}$
	for every sufficiently large $i$. Choose a large
	integer $M$ so that 
	$$\delta_1(n_1+\ldots+n_{N-M})\leq \epsilon n_N$$
	for every $N>M$; 
	this is possible thanks to Lemma~\ref{lem:sum of indices bound}.
	Hence for all sufficiently large $N$, we have:
	$$H(s_N)\leq N\vert\alpha\vert^{n_1+\ldots+n_N+\epsilon n_N+\epsilon(n_{N-M+1}+\ldots+n_N)}\leq \vert \alpha\vert^{n_1+\ldots+n_N+4\epsilon n_N}$$
	and this finishes the proof.
	\end{proof}
	
	\begin{corollary}\label{cor:s is irrational}
	$s$ is irrational.
	\end{corollary}
	\begin{proof}
	Suppose $s$ is rational. From: 
	\begin{itemize}
	\item $\vert s-s_N\vert=O(C_{N+1}/U_{N+1})=O(\vert\alpha\vert^{n_{N+1}+o(n_{N+1})})$, 
	\item $H(s_N)=\vert \alpha\vert^{n_1+\ldots+n_N+o(n_N)}$, 
	\item $n_1+\ldots+n_N<(c-1)n_{N+1}$, and 
	\item $c>2$,
	\end{itemize}
	we have that $s=s_N$ for all sufficiently large $N$, contradiction.
	\end{proof}
	
	\begin{proposition}\label{prop:can't be geq 5}
	 There are only finitely many $k$ such that $\displaystyle\frac{n_{k+1}}{n_k}\geq 5$.
	\end{proposition}
	\begin{proof}
	Let $\epsilon>0$ that will be specified later. Suppose there are infinitely many $k$ such that $n_{k+1}/n_k\geq 5$. For each such $k$ that is sufficiently large,
	we have 
	$$\vert s-s_k\vert= O(C_{k+1}/U_{k+1})=O(\alpha^{-(1-\epsilon)n_{k+1}})=O(\alpha^{-5(1-\epsilon)n_k})$$
	 while 
	 $H(s_k)\leq \vert\alpha\vert^{n_1+\ldots+n_k+\epsilon n_k}$. Note
	 that 
	 $$n_1+\ldots+n_k<\left(1+\frac{1}{c-1}\right)n_k=\frac{c}{c-1}n_k.$$	 
	 We now require $\epsilon$ to satisfy:
	 \begin{equation}\label{eq:epsilon const 1}
	 5(1-\epsilon)n_k>\left(\frac{2c}{c-1}+3\epsilon\right)n_k;
	 \end{equation}
	this is possible since $5>\displaystyle\frac{2c}{c-1}$.
	Then Roth's theorem implies that the $s_k$'s take a single 
	value for infinitely many such $k$. But this contradicts
	\eqref{eq: the s and sk's}.
	\end{proof}
	
	\begin{remark}
	In Proposition~\ref{prop:can't be geq 5}, the same arguments can be used when we replace $5$ by
	any constant greater than $\displaystyle\frac{2c}{c-1}$. When $c>3$, we have $c>\displaystyle\frac{2c}{c-1}$ and this explains the transcendence of $\displaystyle\sum\frac{c_{n_k}}{u_{n_k}}$ given the ``easy'' bound $c>3$.
	\end{remark}

	Note that 
	$\alpha\beta=\pm 1$ since they are units. Then we can use the geometric series to express:
	\begin{align}\label{eq:Ck/Uk geometric series}
	\begin{split}
	\frac{C_k}{U_k}=\frac{C_k}{A_k\alpha^{n_k}(1-(B_k\beta^{n_k})/(A_k\alpha^{n_k}))}&=\sum_{j=0}^\infty\frac{C_k}{A_k\alpha^{n_k}}\left(\frac{B_k(\pm 1)^{n_k}}{A_k\alpha^{2n_k}}\right)^j\\
	&=\sum_{j=0}^\infty\frac{(\pm 1)^{n_kj}C_kB_k^j}{A_k^{j+1}\alpha^{(2j+1)n_k}}
	\end{split}
	\end{align}
	which is valid when $k$ is sufficiently large so that
	$\vert B_k/A_k\vert <\vert(\alpha/\beta)^{n_k}\vert=\vert\alpha\vert^{2n_k}.$
	
	\medskip
	
	Let $P$ be a large positive integers that will be specified later.
	In the following, $N$ denotes an arbitrarily large positive 
	integer. In the various $O$-notations and $o$-notations, the implied constants might depend on the given data and $P$ but they are independent of $N$.  We have:
	\begin{equation}\label{eq:s-s_(N-1) error}
	\vert s-s_{N-1}\vert=\vert\alpha\vert^{-n_N+o(n_N)}.
	\end{equation}
	As mentioned in the Section~\ref{sec:intro}, 
	it is typical in applications of
	the Subspace Theorem to break $s_{N-1}$ as
	$s_{N-P}$ and truncate the expression 
	\eqref{eq:Ck/Uk geometric series} for $k=N-P+1,\ldots,N-1$
	to maintain the error term 
	$\alpha^{-n_N+o(n_N)}$.
	
	\medskip
	
	For $1\leq i\leq P-1$, let
	$$D_{N,i}:=\left\lfloor \frac{n_N}{2n_{N-P+i}}\right\rfloor\leq \frac{5^{P-i}}{2}$$ 
	thanks to Proposition~\ref{prop:can't be geq 5}. The explicit upper bound here is not important: the key fact is that these $D_{N,i}$'s can be bounded from above independently of $N$.
	\begin{proposition}\label{prop:error when truncating}
	For all sufficiently large $N$, we have:
	\begin{equation*}
	\left\vert \frac{C_{N-P+i}}{U_{N-P+i}}-\sum_{j=0}^{D_{N,i}}\frac{(\pm 1)^{n_{N-P+i}j}C_{N-P+i}B_{N-P+i}^j}{A_{N-P+i}^{j+1}\alpha^{(2j+1)n_{N-P+i}}}\right\vert<\vert\alpha\vert^{-n_N+o(n_N)}
	\end{equation*}
	for $i=1,\ldots,P-1$. This means for every $\epsilon>0$, the LHS
	is less than $\vert\alpha\vert^{-n_N+\epsilon n_N}$
	for all sufficiently large $N$.
	\end{proposition}
	\begin{proof}
	Let $\epsilon>0$. We have 
	\begin{align*}
	&\frac{C_{N-P+i}}{U_{N-P+i}}-\sum_{j=0}^{D_{N,i}}\frac{(\pm 1)^{n_{N-P+i}j}C_{N-P+i}B_{N-P+i}^j}{A_{N-P+i}^{j+1}\alpha^{(2j+1)n_{N-P+i}}}\\
	&=\sum_{j=D_{N,i}+1}^{\infty}\frac{(\pm 1)^{n_{N-P+i}j}C_{N-P+i}B_{N-P+i}^j}{A_{N-P+i}^{j+1}\alpha^{(2j+1)n_{N-P+i}}}.
	\end{align*}
	Hence it suffices to require the first term in the RHS:
	$$\frac{(\pm 1)^{n_{N-P+i}j}C_{N-P+i}B_{N-P+i}^j}{A_{N-P+i}^{j+1}\alpha^{(2j+1)n_{N-P+i}}}\ \text{with $j=D_{N,i}+1$}$$ 
	to be $O(\vert\alpha\vert^{-(1-\epsilon/2)n_N})$. This is actually
	the case, as follows. First, by the definition of $D_{N,i}$ we have
	$(2D_{N,i}+3)n_{N-P+i}\geq n_N$. Second
	$\displaystyle\left\vert\frac{C_{N-P+i}B_{N-P+i}^{D_{N,i}+1}}{A_{N-P+i}^{D_{N,i}+2}}\right\vert<\vert\alpha\vert^{(\epsilon/2) n_N}$
	when $N$ is sufficiently large since $D_{N,i}$
	is bounded above independently of $N$ and the assumption
	on the sequences $(A_k)$, $(B_k)$, and $(C_k)$. 
	\end{proof}
	
	At this point, one may attempt to apply the Subspace Theorem 
	using the inequality:
	\begin{align}\label{eq:not yet for Subspace Theorem}
	\left\vert s-s_{N-P}-\sum_{i=1}^{P-1}\sum_{j=0}^{D_{N,i}}\frac{(\pm 1)^{n_{N-P+i}j}C_{N-P+i}B_{N-P+i}^j}{A_{N-P+i}^{j+1}\alpha^{(2j+1)n_{N-P+i}}}\right\vert
	<\vert\alpha\vert^{-n_N+o(n_N)}
	\end{align}
	and linear forms in $2+\displaystyle\sum_{i=1}^{P-1}(D_{N,i}+1)$
	variables for the terms $s$, $s_{N-P}$, and those in the double sum in a similar manner to \cite[p.~180--181]{CZ04_OT} or
	\cite[Proposition~3.4]{KMN19_AA}. However, unlike these previous papers, the term $s_{N-P}$ in our situation is not an $S$-integer
	(for infinitely many $N$) 
	for any choice of a finite set $S\subset M_K$. Because of this,
	there is a potential contribution of
	$\displaystyle H(s_{N-P})^{2+\sum (D_{N,i}+1)}$ which could completely offset the error term $\vert\alpha\vert^{-n_N+o(n_N)}$.
	
	\medskip
	
	Our new idea is to consider an extra ``buffer zone'' 
	by specifying another positive integer $Q<P$, expressing
	\begin{equation}\label{eq:buffer zone}
	\frac{C_{N-P+1}}{U_{N-P+1}}+\ldots+\frac{C_{N-P+Q}}{U_{N-P+Q}}=\frac{x'_N}{x_N}
	\text{with $x_N=\displaystyle\prod_{i=1}^Q U_{N-P+i}$},
	\end{equation} 
	and rewriting \eqref{eq:not yet for Subspace Theorem} as
	\begin{align*}
	\begin{split}
	&\left\vert s-s_{N-P}-\frac{x'_N}{x_N}-\sum_{i=1}^{P-Q-1}\sum_{j=0}^{D_{N,Q+i}}\frac{(\pm 1)^{n_{N-P+Q+i}j}C_{N-P+Q+i}B_{N-P+Q+i}^j}{A_{N-P+Q+i}^{j+1}\alpha^{(2j+1)n_{N-P+Q+i}}}\right\vert\\
	&<\vert\alpha\vert^{-n_N+o(n_N)}.
	\end{split}
	\end{align*}
	We then multiply both sides by $x_N$ to get
	\begin{align}\label{eq:multiply by xN}
	\begin{split}
	&\left\vert x_Ns-x_Ns_{N-P}-x'_N-\sum_{i=1}^{P-Q-1}\sum_{j=0}^{D_{N,Q+i}}x_N\frac{(\pm 1)^{n_{N-P+Q+i}j}C_{N-P+Q+i}B_{N-P+Q+i}^j}{A_{N-P+Q+i}^{j+1}\alpha^{(2j+1)n_{N-P+Q+i}}}\right\vert\\
	&<\vert x_N\vert\vert\alpha\vert^{-n_N+o(n_N)}.
	\end{split}
	\end{align}
	After that we expand $x_N=\displaystyle\prod_{i=1}^Q U_{N-P+i}=\displaystyle\prod_{i=1}^Q(A_{N-P+i}\alpha^{n_{N-P+i}}-B_{N-P+i}\beta^{n_{N-P+i}})$ as a linear
	combination of $2^Q$ terms, expand $x'_N$ as a linear combination of
	$Q2^{Q-1}$ terms. Note that each $x_Ns$, $x_Ns_{N-P}$,
	as well as each individual term in the double sum
	 now consists of $2^Q$ many terms.
	In a typical application of the Subspace Theorem, one is
	worse off after performing the above steps. Therefore it
	is amusing that in our current situation, those steps can help
	reduce the number of terms significantly while the new error
	$\vert x_N\vert\vert\alpha\vert^{-n_N+o(n_N)}$
	is not too much larger than
	the previous 
	$\vert\alpha\vert^{-n_N+o(n_N)}$.
	
	\medskip
	
	Now even if we can apply the Subspace Theorem, there remains one
	important technical issue to overcome. 
	After expanding $x'_N$ and $x_N$,
	it might happen that certain terms in the 
	LHS of \eqref{eq:multiply by xN} already satisfied a linear 
	relation and the conclusion of the Subspace Theorem trivially
	illustrates this fact. For instance, in the double sum in the LHS
	of \eqref{eq:not yet for Subspace Theorem}, if there are two 
	different $(i_1,j_1)$ and $(i_2,j_2)$ for which 
	$(2j_1+1)n_{N-P+i_1}$ and $(2j_2+1)n_{N-P+i_2}$ are close 
	(or even equal)
	to each other then one should ``gather'' 
	the two terms corresponding to $(i_1,j_1)$ and $(i_2,j_2)$ first. 
	So we will also need a way to efficiently ``gather similar terms'' 
	so
	that the conclusion of the Subspace Theorem becomes helpful for our 
	purpose. First, we expand $x_N$ and $x'_N$:
	\begin{lemma}\label{lem:expression for xN and x'N}
	\begin{itemize}
	\item [(i)] There exists $\delta_2>0$ (possibly depending on $P$ and $Q$) such that for 
	all sufficiently large $N$, we can express:
	$$x_N=\sum_{i=1}^{2^Q}x_{N,i}\alpha^{x(N,i)}$$ 
	with the following properties:
	\begin{itemize}
		\item [(a)] $x_{N,i}\in \Q(\alpha)^*$ and $x(N,i)\in\Z$ for every $i$.
		\item [(b)] $x(N,1)=\displaystyle\sum_{i=1}^Q n_{N-P+i}$ and $x(N,1)\geq x(N,j)+2n_{N-P+1}$ for every $j>1$.
		\item [(c)] $\vert x(N,i)\vert\leq x(N,1)$ for every $i$.
		\item [(d)] $h(x_{N,i})/n_N\to 0$ as $N\to\infty$ for every $i$.
		\item [(e)] $\vert x(N,i)-x(N,j)\vert\geq \delta_2 n_N$ for any 
		$1\leq i\neq j\leq 2^Q$.
	\end{itemize}
	\item [(ii)] For all sufficiently large $N$, we can express
	$$x'_N=\sum_{i=1}^{Q2^{Q-1}}x'_{N,i}\alpha^{x'(N,i)}$$
	with the following properties:
	\begin{itemize}
		\item [(a)] $x'_{N,i}\in\Q(\alpha)^*$ and $x'(N,i)\in\Z$ for every $i$.
		\item [(b)] $\vert x'(N,i)\vert\leq n_{N-P+2}+\ldots+n_{N-P+Q}$
		for every $i$.
		\item [(c)] $h(x'_{N,i})/n_N\to 0$ as $N\to\infty$.
	\end{itemize}
	\end{itemize}
	\end{lemma}
	\begin{proof}
	For part (i), let $\scrQ=\{1,\ldots,Q\}$. For each $T\subseteq \scrQ$, put
	$$\Sigma(N,T)=\sum_{i\in T}n_{N-P+i}-\sum_{i\in \scrQ\setminus T}n_{N-P+i}.$$
   Note that $\vert \Sigma(N,T)\vert \leq \displaystyle\sum_{i=1}^Q n_{N-P+i}<2n_{N-P+Q}$
   where the last inequality follows from 
   Lemma~\ref{lem:sum of indices bound}. 
    We have: 
    \begin{align*}
    x_{N}&=\prod_{j=1}^Q (A_{N-P+j}\alpha^{n_{N-P+j}}-B_{N-P+j}\beta^{n_{N-P+j}})\\
    &=\prod_{j=1}^Q(A_{N-P+j}\alpha^{n_{N-P+j}}-B_{N-P+j}(\pm 1)^{n_{N-P+j}}\alpha^{-n_{N-P+j}}).
    \end{align*}
    We fix once and for all a 1-1 correspondence between $\{1,\ldots, 2^Q\}$ and the set of subsets of $\scrQ$ so that $1$ corresponds to
    $\scrQ$.  This allows us to take
    the
     $x(N,i)$'s to be exactly the $\Sigma(N,T)$'s (with
     $x(N,1)=\Sigma(N,\scrQ)$) and the $x_{N,i}$'s are 
     the corresponding products of terms among
     the $A_{N-P+j}$ and $(-1)^{n_{N-P+j}}B_{N-P+j}$; this proves parts 
     (a) and (d).
     The largest among the $\Sigma(N,T)$'s is $\displaystyle\sum_{i=1}^Q n_{N-P+i}$ while the smallest is $-\displaystyle\sum_{i=1}^Q n_{N-P+i}$; this proves part (c). Moreover, the second largest is
     $$-n_{N-P+1}+n_{N-P+2}+\ldots+n_{N-P+Q}$$
     and this proves part (b).
      It remains to prove part (e).
     
     Consider two different subsets $T$ and $T'$ of $\scrQ$. Let $j^*$ be the largest element in
     $T\Delta T'$, then we have:
     \begin{align*}
     \vert \Sigma(N,T)-\Sigma(N,T')\vert&\geq 2n_{N-P+j^*}-\sum_{j<j^*}2n_{N-P+j}\\
     &\geq \frac{2(c-2)}{c-1}n_{N-P+j^*}\\
     &\geq \frac{2(c-2)}{(c-1)5^{P-j^*}}n_N
     \end{align*} 
	where the last two inequalities follow from Lemma~\ref{lem:sum of indices bound} and Proposition~\ref{prop:can't be geq 5}. We can now take
	$\delta_2=\displaystyle\frac{2(c-2)}{5^{P-1}(c-1)}$.
	
	The proof of part (ii) is similar by expanding:
	\begin{align*}
	x'_N&=\sum_{i=1}^Q C_{N-P+i}\prod_{1\leq j\leq Q,\ j\neq i} U_{N-P+j}\\
	&=\sum_{i=1}^Q C_{N-P+i}\prod_{1\leq j\leq Q,\ j\neq i}(A_{N-P+j}\alpha^{n_{N-P+j}}-B_{N-P+j}(\pm\alpha)^{-n_{N-P+j}})
	\end{align*}
	into $Q2^{Q-1}$ many terms.
	\end{proof}
	
	Then we expand each individual term in the double sum
	$$\displaystyle\sum_{i=1}^{P-Q-1}\sum_{j=0}^{D_{N,Q+i}}x_N\frac{(\pm 1)^{n_{N-P+Q+i}j}C_{N-P+Q+i}B_{N-P+Q+i}^j}{A_{N-P+Q+i}^{j+1}\alpha^{(2j+1)n_{N-P+Q+i}}}$$
	to get:
	\begin{lemma}\label{lem:expanding terms in the double sum}
	Put $\eta=2^Q\displaystyle\sum_{i=1}^{P-Q-1}(D_{N,Q+i}+1)$.
	For all sufficiently large $N$, we can express:
	$$\displaystyle\sum_{i=1}^{P-Q-1}\sum_{j=0}^{D_{N,Q+i}}x_N\frac{(\pm 1)^{n_{N-P+Q+i}j}C_{N-P+Q+i}B_{N-P+Q+i}^j}{A_{N-P+Q+i}^{j+1}\alpha^{(2j+1)n_{N-P+Q+i}}}=\sum_{i=1}^\eta y_{N,i}\alpha^{y(N,i)}$$
	with the following properties:
	\begin{itemize}
		\item [(a)] $y_{N,i}\in \Q(\alpha)^*$ and $y(N,i)\in\Z$
		for every $i$.
		
		\item [(b)] $h(y_{N,i})/n_N\to 0$ as $N\to\infty$.
		
		\item [(c)] $y(N,i)\leq x(N,1)-n_{N-P+Q+1}=n_{N-P+1}+\ldots+n_{N-P+Q}-n_{N-P+Q+1}$ for every $i$.
		
		\item [(d)] $y(N,i)>-3n_N$ for every $i$.
	\end{itemize}
	\end{lemma}
	\begin{proof}
	We use the expression for $x_N$ in Lemma~\ref{lem:expression for xN and x'N} to expand each individual term
	in the double sum. This proves (a) and (b). The highest exponent of $\alpha$ in that expression for $x_N$ is $x(N,1)$ while the highest exponent of
	$\alpha$ among the 
	$$\frac{(\pm 1)^{n_{N-P+Q+i}j}C_{N-P+Q+i}B_{N-P+Q+i}^j}{A_{N-P+Q+i}^{j+1}\alpha^{(2j+1)n_{N-P+Q+i}}}$$
	for $1\leq i\leq P-Q-1$ and $0\leq j\leq D_{N,Q+i}$
	is at most $-n_{N-P+Q+1}$, this proves (c). The smallest exponent of $\alpha$ among those terms is
	$$-\max\{(2j+1)n_{N-P+Q+i}:\ 1\leq i\leq P-Q-1,\ 0\leq j\leq D_{N,Q+i}\}>-2n_N$$
	by using the definition of the $D_{N,Q+i}$'s.
	The smallest exponent of  $\alpha$ in $x_N$
	is $-n_{N-P+1}-\ldots-n_{N-P+Q}>-n_N$. This proves (d).
	\end{proof}
	
	We also need the following:
	\begin{lemma}\label{lem: upper bound for xN}
	For all sufficiently large $N$:
	\begin{itemize}
		\item [(i)] $\displaystyle\vert x_N\vert=\vert\alpha\vert^{n_{N-P+1}+\ldots+n_{N-P+Q}+o(n_N)}$, 
		
		\item [(ii)] $\displaystyle\vert x_N(s-s_{N-1})\vert=\left\vert x_N\sum_{k=N}^{\infty}\frac{C_k}{U_k}\right\vert=\vert\alpha\vert^{n_{N-P+1}+\ldots+n_{N-P+Q}-n_N+o(n_N)},$
		
		\item [(iii)] $\displaystyle\vert x_N\vert<\vert\alpha\vert^{\frac{n_N}{c^{P-Q-1}(c-1)}+o(n_N)}$, and 
		
		\item [(iv)] $\vert x_N(s-s_{N-1})\vert<\vert\alpha\vert^{-(1-\frac{1}{c^{P-Q-1}(c-1)})n_N+o(n_N)}.$
	\end{itemize}
	\end{lemma}
	\begin{proof}
	We have $\vert x_N\vert\gg\ll \vert A_{N-P+1}\cdots A_{N-P+Q}\vert\vert\alpha\vert^{n_{N-P+1}+\ldots+n_{N-P+Q}}$ and since
	the $A_k$'s are nonzero with height $o(n_k)$ we have
	$\vert A_{N-P+1}\cdots A_{N-P+Q}\vert=\vert\alpha\vert^{o(n_N)}$. 
	This proves part (i). 
	Then Lemma~\ref{lem:sum of indices bound} gives:
	\begin{equation}\label{eq:n_N/c^(P-Q-1)(c-1)}
	n_{N-P+1}+\ldots+n_{N-P+Q}< \frac{n_N}{c^{P-Q-1}(c-1)}.
	\end{equation}
	This proves part (iii). 
	
	For part (ii), we use part (i) together with:
	$$\vert s-s_{N-1}\vert\gg\ll \vert C_N/U_N\vert = \vert\alpha\vert^{-n_N+o(n_N)}.$$
	Finally part (iv) follows from part (ii) and \eqref{eq:n_N/c^(P-Q-1)(c-1)}.
	\end{proof}
	
	\section{The number $s$ is in $\Q(\alpha)$}\label{sec:s is in Qalpha}
	We continue with the assumption and notation of Section~\ref{sec:prelim}, in particular $s$ is algebraic. Throughout this section, let $Q<P$ be large, yet fixed, positive integers that will be specified later
	and let $N$ denote an arbitrarily large positive 
	integer. In the various $O$-notations and $o$-notations, the implied constants might depend on the given data, $P$, and $Q$ but they are independent of $N$. In this section,
	we finish an important step toward the proof of 
	Theorem~\ref{thm:general} namely proving that $s\in \Q(\alpha)$. 
	This conclusion is similar to the one in the paper of Adamczewski-Bugeaud \cite[Theorem~5]{AB07_OT}. In their paper, to obtain transcendence of the given number \cite[Theorem~5A]{AB07_OT}, they use a result of K.~Schmidt \cite{Sch80_OP}. In this paper, we will need
	more sophisticated applications of the Subspace Theorem together with further 
	combinatorial
	and Galois theoretic arguments in the next section to obtain the desired result.
	
	\medskip
	
	As mentioned in the previous section, before applying the Subspace Theorem, we need to come up with an efficient way to ``gather similar terms'' in the LHS of \eqref{eq:multiply by xN}. This is done first by proving the existence of a certain collection of data then choosing
	a minimal one among those collections.
	\begin{proposition}\label{prop:exist D, E and F}
	Recall the $x_{N,1}$ and $x(N,1)=\displaystyle\sum_{i=1}^Q n_{N-P+i}$ in
	the expression for $x_N$ in 
	Lemma~\ref{lem:expression for xN and x'N}. 
	There exist integers $D,E,F\geq 0$, 
	tuples $(\gamma_1,\ldots,\gamma_E)$ of elements of
	$K$, 
	an infinite 
	set $\scrN$, tuples
	$(d_{N,1},\ldots,d_{N,D})$, $(d(N,1),\ldots,d(N,D))$,
	$(e_{N,1},\ldots,e_{N,E})$, $(e(N,1),\ldots,e(N,E))$,
	$(f_{N,1},\ldots,f_{N,F})$, $(f(N,1),\ldots,f(N,F))$
	for each $N\in\scrN$ with the following properties:
	\begin{itemize}
	\item [(i)] $D+E+F\leq (2^Q-1)+(2^Q-1)+Q2^{Q-1}+2^Q\displaystyle\sum_{i=1}^{P-Q-1}(D_{N,Q+i}+1)$.
	
	\item [(ii)] For every $N\in\scrN$, the $d(N,i)$'s, $e(N,j)$'s, and   
	$f(N,\ell)$'s are integers
	for every $i,j,\ell$.

	\item [(iii)] For every $N\in\scrN$, 
	$n_{N-P+1}+\displaystyle\max_{i,j,\ell}\{d(N,i),e(N,j),f(N,\ell)\}\leq x(N,1)$.
	
	\item [(iv)] For every $N\in\scrN$, $\min_{i,j,k}\{d(N,i),e(N,j),f(N,\ell)\}>-3n_N$.
	
	\item [(v)] For every $N\in\scrN$, the $d_{N,i}$'s, $e_{N,j}$'s,
	and $f_{N,\ell}$'s are elements of $\Q(\alpha)$.
	
	\item [(vi)] As $N\to\infty$ we have 
	$h(d_{N,i})/n_N\to 0$, $h(e_{N,j})/n_N\to 0$,
	and $h(f_{N,\ell})/n_N\to 0$ for every $i,j,\ell$.
	
	\item [(vii)] For all sufficiently large $N\in\scrN$, we have 
	\begin{align}\label{eq:exist DEF, last property}
	\begin{split}
	&\vert sx_{N,1}\alpha^{x(N,1)}+\sum_{j=1}^E \gamma_je_{N,j}\alpha^{e(N,j)}-s_{N-P}x_{N,1}\alpha^{x(N,1)}-
	\sum_{i=1}^D s_{N-P}d_{N,i}\alpha^{d(N,i)}\\
	&-\sum_{\ell=1}^F f_{N,\ell}\alpha^{f(n,\ell)}\vert
	<\vert\alpha\vert^{-(1-\frac{1}{c^{P-Q-1}(c-1)})n_N+o(n_N)}.
	\end{split}
	\end{align}
	\end{itemize} 
	\end{proposition}
	\begin{proof}
	Recall the inequality \eqref{eq:multiply by xN}:
	\begin{align*}
	&\left\vert sx_N-s_{N-P}x_N-x'_N-\sum_{i=1}^{P-Q-1}\sum_{j=0}^{D_{N,Q+i}}x_N\frac{(\pm 1)^{n_{N-P+Q+i}j}C_{N-P+Q+i}B_{N-P+Q+i}^j}{A_{N-P+Q+i}^{j+1}\alpha^{(2j+1)n_{N-P+Q+i}}}\right\vert\\
	&<\vert x_N\vert\vert\alpha\vert^{-n_N+o(n_N)})<\vert\alpha\vert^{-(1-\frac{1}{c^{P-Q-1}(c-1)})n_N+o(n_N)}
	\end{align*}
	where the last equality follows from Lemma~\ref{lem: upper bound for xN} and holds when $N$ is sufficiently large. We now choose $\scrN$ to be the set of all sufficiently large
	integers, $D=E=2^Q-1$, and $\gamma_1=\ldots=\gamma_E=s$. 
	We want the sum
	$$sx_{N,1}\alpha^{x(N,1)}+\sum_{j=1}^E \gamma_je_{N,j}\alpha^{e(N,j)}=s\left(x_{N,1}\alpha^{x(N,1)}+\sum_{j=1}^E e_{N,j}\alpha^{e(N,j)}\right)$$ 
	to be $sx_N$; therefore we simply
	choose the $e_{N,j}$'s and $e(N,j)$'s for $1\leq j\leq E$
	to be respectively the terms $x_{N,k}$'s and $x(N,k)$'s for $2\leq k\leq 2^Q$
	in the expression for $x_N$ in Lemma~\ref{lem:expression for xN and x'N}.
	
	Similarly, we want the sum
	$$s_{N-P}x_{N,1}\alpha^{x(N,1)}+
	\sum_{i=1}^D s_{N-P}d_{N,i}\alpha^{d(N,i)}=s_{N-P}\left(x_{N,1}\alpha^{x(N,1)}+\sum_{i=1}^D d_{N,i}\alpha^{d(N,i)}\right)$$
	to be $s_{N-P}x_N$; therefore we simply choose the
	$d_{N,i}$'s and $d(N,i)$'s for $1\leq i\leq D$
	to be respectively the terms $x_{N,k}$'s and $x(N,k)$'s
	for $2\leq k\leq 2^Q$ above.

	Finally, choose $F=Q2^{Q-1}+\eta$ (with $\eta$ in Lemma~\ref{lem:expanding terms in the double sum}) and we want the sum
	$\displaystyle\sum_{\ell=1}^F f_{N,\ell}\alpha^{f(n,\ell)}$
	to be 
	$$x'_N+\sum_{i=1}^{P-Q-1}\sum_{j=0}^{D_{N,Q+i}}x_N\frac{(\pm 1)^{n_{N-P+Q+i}j}C_{N-P+Q+i}B_{N-P+Q+i}^j}{A_{N-P+Q+i}^{j+1}\alpha^{(2j+1)n_{N-P+Q+i}}}$$
	using the expressions for $x'_N$ and 
	and the double sum given in 
	Lemma~\ref{lem:expression for xN and x'N} and
	Lemma~\ref{lem:expanding terms in the double sum}.
	All the properties (i)-(vii) follow from
	our choice and properties of the expressions for $x_N$, $x'_N$,
	and the double sum given in the previous section.
	\end{proof}
	
	\begin{remark}
	In the proof of Proposition~\ref{prop:exist D, E and F}, we have
	that $D+E+F$ is \emph{exactly} the RHS of (i) and the $\gamma_i$'s
	are \emph{exactly} $s$. However, relaxing these as in 
	Proposition~\ref{prop:exist D, E and F} allows us to work with 
	more possible
	collections of data in order to choose a minimal one.
	\end{remark}
	
	\begin{remark}
	Note that we allow any (or even all) of the $D,E,F$ to be $0$ in
	the statement of Proposition~\ref{prop:exist D, E and F}.
	For example, if $D=E=F=0$ then all the tuples are empty,
	the properties (i)-(vi) are vacuously true, and
	property (vii) becomes: 
	$$\displaystyle\left\vert (s-s_{N-P})x_{N,1}\alpha^{x(N,1)}\right\vert
	<\vert\alpha\vert^{-(1-\frac{1}{c^{P-Q-1}(c-1)})n_N+o(n_N)}.$$
	\end{remark}
	
	In the proof of Proposition~\ref{prop:exist D, E and F},
	we prove the existence of the required data
	by crudely expanding out terms in $x_N$, $x'_N$, and those in the 
	double sum without any simplification whatsoever. 
	The \emph{key trick} is the following:
	
	\begin{definition}\label{def:minimal D+E+F}
	Among all the collections of data $(D,E,F,\scrN,\ldots)$ satisfying properties (i)-(vii) in Proposition~\ref{prop:exist D, E and F}, we choose one
	for which $D+E+F$ is minimal. By abusing the notation, we
	still use the same notation $D$, $E$, $F$, $\scrN$, $\gamma_j$'s, $d_{N,i}$'s, $d(N,i)$'s, $e_{N,j}$'s, $e(N,j)$'s, $f_{N,\ell}$'s, 
	and $f(N,\ell)$'s for
	this chosen data with minimal $D+E+F$.
	\end{definition}

	\begin{remark}
	This trick is similar to the one in \cite[Proposition~3.4]{KMN19_AA} in which the authors worked with a vector space with the minimal dimension among a certain family of finite-dimensional vector spaces so that any further non-trivial linear relation would not be possible.
	\end{remark}
	
	\begin{lemma}\label{lem:D+E+F terms are non-zero}
	There are at most finitely many $N$ in $\scrN$ such that
	one of the terms 
	$s_{N-P}x_{N,i}\alpha^{x(N,i)}$,
	$\gamma_je_{N,j}\alpha^{e(N,j)}$, $f_{N,\ell}\alpha^{f(N,\ell)}$
	for $1\leq i\leq D$, $1\leq j\leq E$, and $1\leq \ell\leq F$
	is zero.
	\end{lemma}
	\begin{proof}
	If there is a term that is zero for an infinite subset $\scrN'$
	of $\scrN$, then we have a new collection of data in which
	$\scrN$ is replaced by $\scrN'$ and that zero term is removed. 
	This violates the minimality of $D+E+F$.
	\end{proof}
	
	The point of Definition~\ref{def:minimal D+E+F} is that
	any non-trivial linear relation among the 
	$x_{N,1}\alpha^{x(N,1)}$, $s_{N-P}x_{N,1}\alpha^{x(N,1)}$,
	$s_{N-P}d_{N,i}\alpha^{d(N,i)}$, $e_{N,j}\alpha^{e(N,j)}$,
	and $f_{N,\ell}\alpha^{f(N,\ell)}$ that holds for infinitely
	many $N$ must involve the first 2 terms. 
	
	\begin{proposition}\label{prop:non-trivial relation must involve..}
	Suppose there exist an infinite subset $\scrN'$ of $\scrN$ and complex numbers
	$\lambda_1$, $\lambda_2$, $\tilde{d}_i$, $\tilde{e}_j$, $\tilde{f}_{\ell}$ for $1\leq i\leq D$, $1\leq j\leq E$, and $1\leq \ell\leq F$
	not all of which are zero such that:
	\begin{align}\label{eq:non-trivial relation 1}
	\begin{split}	
	&\lambda_1x_{N,1}\alpha^{x(N,1)}+\lambda_2s_{N-P}x_{N,1}\alpha^{x(N,1)}+\sum_{i=1}^D\tilde{d}_is_{N-P}d_{N,i}\alpha^{d(N,i)}\\
	&+\sum_{j=1}^E\tilde{e}_je_{N,j}\alpha^{e(N,j)}
	+\sum_{\ell=1}^F\tilde{f}_{\ell}f_{N,\ell}\alpha^{f(N,\ell)}=0
	\end{split}
	\end{align}
	for every $N\in\scrN'$. We have:
	\begin{itemize}
		\item [(i)] There exist $\kappa_1$, $\kappa_2$, 
		$\tilde{d}'_i$,
		$\tilde{e}'_j$, $\tilde{f}'_{\ell}$ for 
		$1\leq i\leq D$, $1\leq j\leq E$, and $1\leq \ell\leq F$
		not all of which are zero
		with the following properties:
		\begin{itemize}
			\item [(a)] All the $\kappa_1$, $\kappa_2$, $\tilde{d}'_i$,
		$\tilde{e}'_j$, and $\tilde{f}'_{\ell}$ are in $\Q(\alpha)$. 
			
			\item [(b)] For every $N\in\scrN'$:
	\begin{align}\label{eq:non-trivial relation 2}
	\begin{split}
	&\kappa_1x_{N,1}\alpha^{x(N,1)}+\kappa_2s_{N-P}x_{N,1}\alpha^{x(N,1)}+\sum_{i=1}^D\tilde{d}'_is_{N-P}d_{N,i}\alpha^{d(N,i)}\\
	&+\sum_{j=1}^E\tilde{e}'_je_{N,j}\alpha^{e(N,j)}
	+\sum_{\ell=1}^F\tilde{f}'_{\ell}f_{N,\ell}\alpha^{f(N,\ell)}=0
	\end{split}
	\end{align}
		
			\item [(c)] $\kappa_1\kappa_2\neq 0$.	
		\end{itemize}
		\item [(ii)] $s\in\Q(\alpha)$.
	\end{itemize}
	\end{proposition}
	\begin{proof}
	For part (i), since the terms $x_{N,1}\alpha^{x(N,1)}$, $s_{N-P}x_{N,1}\alpha^{x(N,1)}$,
	$s_{N-P}d_{N,i}\alpha^{d(N,i)}$, $e_{N,j}\alpha^{e(N,j)}$,
	and $f_{N,\ell}\alpha^{f(N,\ell)}$ are in $\Q(\alpha)$,
	this establishes the existence of the
	$\lambda_1$, $\lambda_2$,
	$\tilde{d}'_i$, $\tilde{e}'_j$, $\tilde{f}'_{\ell}$
	satisfying properties (a) and (b). We now prove
	$\kappa_1\kappa_2\neq 0$.
	
	First, assume that $\kappa_1=\kappa_2=0$. This means:
	\begin{align}\label{eq:non-trivial relation 3}
	\sum_{i=1}^D\tilde{d}'_is_{N-P}d_{N,i}\alpha^{d(N,i)}
	+\sum_{j=1}^E\tilde{e}'_je_{N,j}\alpha^{e(N,j)}
	+\sum_{\ell=1}^F\tilde{f}'_{\ell}f_{N,\ell}\alpha^{f(N,\ell)}=0
	\end{align}
	for $N\in\scrN'$. Now assume that $\tilde{d}'_{i^*}\neq 0$
	for some $i^*$. Then for $N\in\scrN'$, equation \eqref{eq:non-trivial relation 3} allows
	us to express $s_{N-P}d_{N,i^*}\alpha^{d(N,i^*)}$
	as a linear combination of the
	$s_{N-P}d_{N,i}\alpha^{d(N,i)}$ with $i\neq i^*$,
	the $e_{N,j}\alpha^{e(N,j)}$, and the
	$f_{N,\ell}\alpha^{f(N,\ell)}$ with coefficients
	in $\Q(\alpha)$. This allows us to come up with a new data 
	satisfying the properties in 
	Proposition~\ref{prop:exist D, E and F}
	in which $\scrN$ is replaced by $\scrN'$ and $D$ is replaced by 
	$D-1$. This contradicts the minimality of $D+E+F$.
	Therefore $d_{N,i}=0$ for  $1\leq i\leq D$ and every $N\in\scrN'$.
	Similarly $f_{N,\ell}=0$ for $1\leq \ell\leq F$ and every $N\in\scrN'$.
	So now \eqref{eq:non-trivial relation 3} becomes:
	$$\sum_{j=1}^E\tilde{e}'_je_{N,j}\alpha^{e(N,j)}=0\ \text{for $N\in\scrN'$.}$$
	Arguing as before, we obtain a contradiction to the minimality
	of $D+E+F$. This proves at least $\kappa_1$ or $\kappa_2$
	is non-zero.
	
	We emphasize that the arguments should be run 
	in the above order (i.e. obtaining $d_{N,i}=f_{N,\ell}=0$ first). 
    Suppose one tried to prove all the $e_{N,j}=0$ \emph{first} by 
    using \eqref{eq:non-trivial relation 3} to express some
    $e_{N,j^*}\alpha^{e(N,j^*)}$ as a linear combination
    of the $s_{N-P}d_{N,i}\alpha^{d(N,i)}$, the $e_{N,j}\alpha^{e(N,j)}$ with $j\neq j^*$, and the $f_{N,\ell}\alpha^{f(N,\ell)}$.
    Then due to the term $\gamma_{j^*}e_{N,j^*}\alpha^{e(N,j^*)}$
    in the LHS of \eqref{eq:exist DEF, last property} and
    since at the moment we do \emph{not} necessarily have
    $\gamma_{j^*}\in\Q(\alpha)$, the ``new'' $d_{N,i}$
    and $f_{N,\ell}$ would not remain in $\Q(\alpha)$ and the new data would not satisfy all the properties in Proposition~\ref{prop:exist D, E and F}. 
    
    Suppose $\kappa_1\neq 0$ while $\kappa_2=0$. We have
    $$\kappa_1x_{N,1}\alpha^{x(N,1)}+\sum_{i=1}^D\tilde{d}'_is_{N-P}d_{N,i}\alpha^{d(N,i)}
	+\sum_{j=1}^E\tilde{e}'_je_{N,j}\alpha^{e(N,j)}
	+\sum_{\ell=1}^F\tilde{f}'_{\ell}f_{N,\ell}\alpha^{f(N,\ell)}=0$$
    for $N\in\scrN'$. We divide by $x_{N,1}\alpha^{x(N,1)}$
    and
    note that each of $d(N,i)-x(N,1)$, $e(N,j)-x(N,1)$,
    and $f(N,\ell)-x(N,1)$ is
    at most $-n_{N-P+1}\leq -n_N/5^{P-1}$
    by (iii) in Proposition~\ref{prop:exist D, E and F}
    and Proposition~\ref{prop:can't be geq 5} while
    each of $\vert s_{N-P}d_{N,i}/x_{N,1}\vert$, 
    $\vert e_{N,j}/x_{N,1}\vert$,
    and $\vert f_{N,\ell}/x_{N,1}\vert$
    is $e^{o(n_N)}$. Therefore taking limit as $N\to\infty$ (and $N\in\scrN'$), we get $\kappa_1=0$, contradiction. The case $\kappa_2\neq 0$ and $\kappa_1=0$ is ruled out in a similar way. This finishes the proof of (c).
    
    For part (ii), we use \eqref{eq:non-trivial relation 2}, divide
    by $x_{N,1}\alpha^{x(N,1)}$ and let $N\to\infty$ as above to obtain
    $$\kappa_1+\kappa_2s=0$$
    and this proves $s=-\kappa_1/\kappa_2\in \Q(\alpha)$.  
	\end{proof}
	
	\begin{proposition}\label{prop:s in Qalpha}
	The number $s$ is in $\Q(\alpha)$.
	\end{proposition}
	\begin{proof}
	We will obtain a non-trivial linear relation
	as in the statement of 
	Proposition~\ref{prop:non-trivial relation must involve..}
	for infinitely  many $N$ and apply part (ii).
	
	Let $v_{\infty}$ be the valuation on $\Q(\alpha)$ corresponding to the 
	usual $\vert\cdot\vert$ and let $w$ be the other 
	archimedean one. Note that we follow the normalization
	in \cite[Chapter~1]{BG06_HI}, hence:
	$$\vert x\vert_{v_\infty}=\vert x\vert^{1/2}\ \text{and}\ \vert x\vert_w=\vert\sigma(x)\vert^{1/2}.$$
	The archimedean valuations on $K$ are denoted as $v_1,\ldots,v_m$
	and $w_1,\ldots,w_m$ where the $v_i$'s lie above $v_{\infty}$ and the 
	$w_i$'s lie above $w$. They correspond to the following real
	or one for each pair of complex-conjugate embeddings of 
	$K$ into $\C$: $\tau_1,\ldots,\tau_m$ and $\sigma_1,\ldots,\sigma_m$. In other words:
	$$\vert x\vert_{v_i}=\vert\tau_i(x)\vert^{d(v_i)/[K:\Q]}\ \text{and}\
	  \vert x\vert_{w_i}=\vert\sigma_i(x)\vert^{d(w_i)/[K:\Q]}$$
	  where $d(v_i)=[K_{v_i}:\R]=1$ or $2$ depending on whether $v_i$ is real or complex and a similar definition for $d(w_i)$. Note that the
	  $\tau_i$'s restrict to the identity automorphism 
	  on $\Q(\alpha)$ while
	  the $\sigma_i$'s restrict to $\sigma$ on $\Q(\alpha)$.
	 In fact, since $K/\Q$ is Galois, either all archimedean valuations
	 are real or all are complex and we simply let $\delta$ be the common value of the $d(v_i)$ and $d(w_i)$. We have:
	  \begin{equation}\label{eq:sum of local degrees}
	  \sum_{i=1}^m d(v_i)=\sum_{i=1}^m d(w_i)=m\delta=[K:\Q(\alpha)]=\frac{[K:\Q]}{2}.
	  \end{equation}
	  
	 Our next step is to apply the Subspace Theorem using \eqref{eq:exist DEF, last property}. 
	 Fix $\epsilon>0$ that will be specified later. 
	 Let 
	 $$S=M_K^{\infty}=\{v_1,\ldots,v_m,w_1,\ldots,w_m\}.$$
	 We will work with linear forms in variables
	 $$(T_1,T_2,X_1,\ldots,X_D,Y_1,\ldots,Y_E,Z_1,\ldots,Z_F)$$
	 and the vectors
	 \begin{align*} 
	 \bfv_N=(&x_{N,1}\alpha^{x(N,1)},-s_{N-P}x_{N,1},(-s_{N-P}d_{N,i}\alpha^{d(N,i)})_{1\leq i\leq D},(e_{N,j}\alpha^{e(N,j)})_{1\leq j\leq E},\\
	 &(-f_{N,\ell}\alpha^{f(N,\ell)})_{1\leq \ell\leq F})
	 \end{align*}
	 for $N\in\scrN$. 
	
	For each $v\in S$, the linear forms are denoted: $L_{v,T,1}$, $L_{v,T,2}$, $L_{v,X,i}$ for $1\leq i\leq D$, $L_{v,Y,j}$ for $1\leq j\leq E$,
	and $L_{v,Z,\ell}$ for $1\leq \ell\leq F$. The reason is
	that they will be defined as follows:
	\begin{itemize}
		\item For any $v\in S$, $L_{v,T,2}=T_2$, $L_{v,X,i}=X_i$,
		$L_{v,Y,j}=Y_j$, and $L_{v,Z,\ell}=Z_{\ell}$ for any $i,j,\ell$.
		
		\item If $v=w_k$ for some $1\leq k\leq m$, define $L_{v,T,1}=T_1$ and we have $$\vert L_{v,T,1}(\bfv_N)\vert_v=\vert \sigma(x_{N,1}\alpha^{x(N,1)})\vert^{\delta/[K:\Q]}=o(1)$$
		since $\vert\sigma(x_{N,1})\vert=e^{o(n_N)}$
		while $\vert\sigma(\alpha^{x(N,1)})\vert=\vert\alpha\vert^{-x(N,1)}$ and $x(N,1)\gg n_N$.
		
		\item If $v=v_k$ for some $1\leq k\leq m$, define:
		\begin{align*}
		L_{v,T,1}=&\tau_k^{-1}(s)T_1+T_2+X_1+\ldots+X_D+\tau_k^{-1}(\gamma_1)Y_1+\ldots+\tau_k^{-1}(\gamma_E)Y_E\\
		&+Z_1+\ldots+Z_F
		\end{align*}
		so that 
		$\vert L_{v,T,1}(\bfv_N)\vert_v$
		is exactly the LHS of \eqref{eq:exist DEF, last property} to the power ${\delta/[K:\Q]}$. Therefore 
		$$\vert L_{v,T,1}(\bfv_N)\vert_v
		<\vert\alpha\vert^{-(1-\frac{1}{c^{P-Q-1}(c-1)}-\epsilon)n_N\delta/[K:\Q]}$$
		for all sufficiently large $N\in\scrN$.
	\end{itemize}
	
	Combining with ~\eqref{eq:sum of local degrees}, we have:
	\begin{equation}\label{eq:apply ST0}
	\prod_{v\in S}\vert L_{v,T,1}(\bfv_N)\vert_v=O(\vert\alpha\vert^{-(1-\frac{1}{c^{P-Q-1}(c-1)}-\epsilon)\frac{n_N}{2}}).
	\end{equation}
	Now since $\alpha$ is an $S$-unit (i.e. usual algebraic integer unit), by the product
	formula together with the fact that $\vert s_{N-P}\vert=O(1)$ and the
	$x_{N,1}$, $d_{N,i}$, $e_{N,j}$, $f_{N,\ell}$ have height $o(n_N)$,
	we have that for all sufficiently large $N\in\scrN$:
	\begin{equation}\label{eq:apply ST 1}
	\prod_{v\in S}\prod_{L}\vert L(\bfv_N)\vert_v=O(\vert\alpha\vert^{-(1-\frac{1}{c^{P-Q-1}(c-1)}-2\epsilon)\frac{n_N}{2}})
	\end{equation}
	where $L$ ranges over all the $L_{v,T,1}$, $L_{v,T,2}$, $L_{v,X,i}$, $L_{v,Y,j}$, and $L_{v,Z,\ell}$. Recall that 
	$\tilde{\bfv}_N$ denotes the 
	point in the projective space with coordinates
	$\bfv_N$. Since the $\vert d(N,i)\vert$, $\vert e(N,j)\vert$,
	and $\vert f(N,\ell)\vert$ are less than $3n_N$, we have
	$$H(\tilde{\bfv}_N)\leq \vert\alpha\vert^{5n_N}$$
	for every $N\in\scrN$. 
	
	We need to obtain $\epsilon'>0$
	such that:
	\begin{equation}\label{eq:need to obtain epsilon'}
	\prod_{v\in S}\prod_{L}\frac{\vert L(\bfv_N)\vert_v}{\vert\bfv_N\vert_v}< H(\tilde{\bfv}_N)^{-D-E-F-2-\epsilon'}
	\end{equation}
	for all large $N$ in $\scrN$.
	We have:
	\begin{equation}\label{eq:apply ST 2}
	H(\tilde{\bfv}_N)^{D+E+F+2}\prod_{v\in S}\prod_{L}\frac{1}{\vert\bfv_N\vert_v}=\left(\prod_{v\in M_K\setminus S}\vert\bfv_N\vert_v\right)^{D+E+F+2}.
	\end{equation} 	
 	Since all the $x_{N,1}$, $d_{N,i}$, $e_{N,j}$, and $f_{N,\ell}$
 	have multiplicative height 
 	$e^{o(n_N)}$, we have:
 	\begin{equation}
 	\left(\prod_{v\in M_K\setminus S}\vert\bfv_N\vert_v\right)^{D+E+F+2}\leq H(s_{N-P})^{D+E+F+2}\vert\alpha\vert^{\epsilon n_N}
 	\end{equation}
 	for all sufficiently large $N\in\scrN$. We have:
 	$$H(s_{N-p})\leq \vert\alpha\vert^{n_1+\ldots+n_{N-P}+o(n_{N-P})}\leq \vert\alpha\vert^{n_{N-P+1}}$$
 	for all large $N\in\scrN$. From Proposition~\ref{prop:exist D, E and F} and the definition of the $D_{N,Q+i}$'s, we have:
 	\begin{align*}
 	D+E+F+2&\leq 2^{Q+1}+Q2^{Q-1}+2^Q\sum_{i=1}^{P-Q-1}(D_{N,Q+i}+1)\\
		&\leq 2^Q(P-(Q/2)+1)+2^Q\sum_{i=1}^{P-Q-1}\frac{n_N}{2n_{N-P+Q+i}}\\
		&\leq 2^QP+2^{Q-1}\sum_{i=1}^{P-Q-1}\frac{n_N}{n_{N-P+Q+i}} 
 	\end{align*}
	assuming $Q\geq 2$.
	Hence for all large $N\in\scrN$,  
	\begin{equation}\label{eq:apply ST 3}
	H(s_{N-P})^{D+E+F+2}\leq \vert\alpha\vert^{\Omega_N}
	\end{equation}
	where
	\begin{align}\label{eq:bounding OmegaN}
	\begin{split}
	\Omega_N:&=n_{N-P+1}2^QP+2^{Q-1}\left(\sum_{i=1}^{P-Q-1}\frac{n_{N-P+1}}{n_{N-P+Q+i}}\right)n_N\\
	&\leq \frac{2^QP}{c^{P-1}}n_N+2^{Q-1}\left(\sum_{i=1}^{P-Q-1}\frac{1}{c^{Q+i-1}}\right)n_N\\
	&\leq \left(\frac{2^QP}{c^{P-1}}+\frac{2^{Q-1}}{c^Q}\cdot\frac{1}{1-(1/c)}\right)n_N.	
	\end{split}
	\end{align}
	
	Finally, note that:
	\begin{equation}\label{eq:apply ST 4}
	H(\tilde{\bfv}_N)^{-\epsilon'}\geq \vert\alpha\vert^{-5\epsilon'n_N}.
	\end{equation}
	In order to obtain $\epsilon'$ satisfying \eqref{eq:need to obtain epsilon'}, we combine \eqref{eq:apply ST 1} and \eqref{eq:apply ST 2}--\eqref{eq:apply ST 4} and require that:
	\begin{equation}\label{eq:s in Qalpha, cond on P,Q,epsilon 1}
	\left(\frac{2^QP}{c^{P-1}}+\frac{2^{Q-1}}{c^Q}\cdot\frac{1}{1-(1/c)}\right)+\epsilon-\frac{1}{2}\left(1-\frac{1}{c^{P-Q-1}(c-1)}-2\epsilon\right)<-5\epsilon'.
	\end{equation}
	Such an $\epsilon'$ exists as long as the LHS is negative. 
	Therefore, at the beginning of Section~\ref{sec:prelim}, we choose
	sufficiently large integers $2\leq Q<P$ and here we choose a 
	sufficiently small $\epsilon>0$ such that:
	\begin{equation}\label{eq:s in Qalpha, cond on P,Q,epsilon 2}
	\frac{2^QP}{c^{P-1}}+\frac{2^{Q-1}}{c^Q}\cdot\frac{1}{1-(1/c)}+2\epsilon+\frac{1}{2c^{P-Q-1}(c-1)}<\frac{1}{2};
	\end{equation}
	this is possible since $c>2$. Then we can apply the Subspace Theorem to have that there exists a non-trivial linear relation satisfied by  the coordinates
	of $\bfv_N$ for infinitely many $N\in\scrN$. Then we use part (ii) of Proposition~\ref{prop:non-trivial relation must involve..} to finish the proof.
 	\end{proof}
	
	\section{The proof of Theorem~\ref{thm:general}}
	We continue with the notations of Section~\ref{sec:prelim} and 
	\emph{ignore} those in Section~\ref{sec:s is in Qalpha}. We no longer assume the choice of $P$, $Q$, and $\epsilon$ as in
	\eqref{eq:s in Qalpha, cond on P,Q,epsilon 1} and 
	\eqref{eq:s in Qalpha, cond on P,Q,epsilon 2}. However, we use the 
	crucial result
	that $s\in\Q(\alpha)$. 
	While the arguments in Section~\ref{sec:s is in Qalpha} are
	valid for any sufficiently large $Q<P$ (and sufficiently small $\epsilon$), 
	those in this section require that $Q=P-1$:
	
	\begin{assumption}
	From now on, $Q=P-1$. Therefore $x_N=\displaystyle\prod_{i=1}^{P-1}U_{N-P+i}$, the $x(N,i)'s$ are 
	the numbers
	$\pm n_{N-P+1}\pm n_{N-P+2}\cdots\pm n_{N-1}$, $\displaystyle\frac{x'_N}{x_N}=\sum_{i=1}^{P-1}\frac{C_{N-P+i}}{U_{N-P+i}}$, and most importantly:
	\begin{equation}\label{eq:important when Q=P-1}
		(s-s_{N-P})x_N-x'_N=(s-s_{N-1})x_N=x_N\sum_{k=N}^\infty\frac{C_k}{U_k}.
	\end{equation}
	\end{assumption}
	
	The following numbers $x(N,+)$, $x(N,-)$, $\tilde{x}(N,+)$, and
	$\tilde{x}(N,-)$ will play an important role:
	\begin{lemma}\label{lem:introducing x(N,+)...}
	\begin{itemize}
	\item [(i)] $x(N,+):=-n_{N-P+1}-\ldots-n_{N-2}+n_{N-1}$ is the smallest
	non-negative numbers while
	$x(N,-):=-x(N,+)$ is the largest non-positive numbers among
	the $x(N,i)$'s.
	
	\item [(ii)] Write $\tilde{x}(N,+):=-n_{N-P+1}-\ldots-n_{N-1}+n_N$
	and write $\tilde{x}(N,-):=-\tilde{x}(N,+)$. We have
	$\tilde{x}(N,-)<x(N,-)<x(N,+)<\tilde{x}(N,+)$. Moreover:
	$$x(N,+)-x(N,-)\geq \frac{2(c-2)}{c-1}n_{N-1}\geq \frac{2(c-2)}{5(c-1)}n_N,\ \text{and}$$
	$$x(N,-)-\tilde{x}(N,-)=\tilde{x}(N,+)-x(N,+)=n_N-2n_{N-1}\geq \frac{c-2}{c}n_N$$
	for all sufficiently large $N$.
	
	\item [(iii)] $\vert (s-s_{N-1})x_N\vert=\vert\alpha\vert^{\tilde{x}(N,-)+o(n_N)}$ for all sufficiently large $N$.
	\end{itemize}
	\end{lemma}
	\begin{proof}
	We have:
	\begin{align}
	\begin{split}
		x(N,+)-x(N,-)&=2(n_{N-1}-n_{N-2}-\ldots-n_{N-P+1})\\
		&\geq \frac{2(c-2)}{c-1}n_{N-1}\geq \frac{2(c-2)}{5(c-1)}n_N
	\end{split}
	\end{align}
	for all large $N$ by Lemma~\ref{lem:sum of indices bound} and 
	Proposition~\ref{prop:can't be geq 5}. Part (i) and the 
	rest of part (ii) are elementary.
	Part (iii) is simply
	Lemma~\ref{lem: upper bound for xN}(ii) when $Q=P-1$.
	\end{proof}
	
	First, we prove
	the existence of a certain expression and then choose a minimal one:
	
	\begin{proposition}\label{prop:exist D and F}
	Note that $Q=P-1$. There exist integers $D,F\geq 0$, an infinite set $\scrN$,
	tuples $(d_{N,1},\ldots,d_{N,D})$, $(f_{N,1},\ldots,f_{N,F})$, $(d(N,1),\ldots,d(N,D))$ and $(f(N,1),\ldots,f(N,F))$ for each $N\in\scrN$ with the following properties:
	\begin{itemize}
		\item [(i)] $D+F\leq 2^{P-1}+(P-1)2^{P-2}$.
		\item [(ii)] For every $N\in\scrN$, the $d(N,i)$'s and $f(N,j)$'s are integers for every $i,j$.
		\item [(iii)] For every $N\in\scrN$,  $\max_i \vert d(N,i)\vert\leq n_{N-P+1}+\ldots+n_{N-1}$. 
		\item [(iv)] For every $N\in\scrN$, $\max_j \vert f(N,j)\vert\leq n_{N-P+2}+\ldots+n_{N-1}$.
		\item [(v)] For every $N\in\scrN$, the $d_{N,i}$'s 
		and $f_{N,j}$'s are elements of $\Q(\alpha)$.
		\item [(vi)] As $N\to\infty$, we have $h(d_{N,i})/n_N\to 0$
	     and $h(f_{N,j})/n_n\to 0$ for every $i,j$.
		\item [(vii)] For every $N\in\scrN$, we have
		\begin{align}\label{eq:exist DF last property}	
		\sum_{i=1}^D(s-s_{N-P})d_{N,i}\alpha^{d(N,i)}-\sum_{j=1}^F f_{N,j}\alpha^{f(N,j)}
		=(s-s_{N-1})x_N.
		\end{align}
	\end{itemize}
	\end{proposition}
	\begin{proof}
	The proof is very similar to the proof
	of Proposition~\ref{prop:exist D, E and F}. 
	Choose $\scrN$ to be the set of all sufficiently large integers
	and
	$D=2^{P-1}$. We want the sum
	$\displaystyle\sum_{i=1}^D(s-s_{N-P})d_{N,i}\alpha^{d(N,i)}$
	to be $(s-s_{N-P})x_N$. We simply choose the
	$d_{N,i}$ and $d(N,i)$ 
	to be the $x_{N,i}$ and $x(N,i)$. 
	
	Then we choose $F=(P-1)2^{P-2}$  and we want the sum $\displaystyle\sum_{j=1}^F f_{N,j}\alpha^{f(n,j)}$
	to be 
	$$x'_N=\sum_{j=1}^{(P-1)2^{P-2}}x'_{N,j}\alpha^{x'(N,j)},$$
	so we simply take the $f_{N,j}$ and $f(N,j)$ to be
	the $x'_{N,j}$ and $x(N,j)$. The properties (i)-(vi)
	follow from \eqref{eq:important when Q=P-1}
	and Lemma~\ref{lem:expression for xN and x'N}.
	\end{proof}

	\begin{definition}\label{def:minimal D+F}
	Among all the collections of data $(D,F,\scrN,\ldots)$ satisfying properties (i)-(vii) in Proposition~\ref{prop:exist D and F}, we choose one
	for which $D+F$ is minimal. By abusing the notation, we
	still use the same notation $D$, $F$, $\scrN$, $d_{N,i}$'s, $d(N,i)$'s, $f_{N,j}$'s, 
	and $f(N,j)$'s for
	this chosen data with minimal $D+F$.
	\end{definition}
	
	\begin{remark}
	As before, we allow the possibility that $D$ or $F$ is $0$. Note
	that the scenario $D=F=0$ cannot happen since the RHS of
	\eqref{eq:exist DF last property} is non-zero.
	\end{remark}
	
	\begin{lemma}\label{lem:D+F terms are nonzero}
	There are at most finitely many $N$ in $\scrN$ such that one of the
	terms $(s-s_{N-P})d_{N,i}\alpha^{d(N,i)}$,
	$f_{N,j}\alpha^{f(N,j)}$ for $1\leq i\leq D$
	and $1\leq j\leq F$ is zero.
	\end{lemma}
	\begin{proof}
	This is similar to the proof of Lemma~\ref{lem:D+E+F terms are non-zero}.
	\end{proof}
	
	The reason we introduce the number $\tilde{x}(N,-)$ is 
	that the $d(N,i)$ for $1\leq i\leq D$ and $f(N,j)$ for $1\leq j\leq F$ are less than $\tilde{x}(N,-)+o(n_N)$ for an appropriate choice of $P$. This means that for any $\epsilon>0$, as long as $P$ is sufficiently large (depending on $\epsilon$), then the $d(N,i)$'s and $f(N,j)$'s are
	smaller than
	$\tilde{x}(N,-)+\epsilon n_N$. For our purpose, we state and
	prove the next result for the specific 
	value $\displaystyle\frac{c-2}{2c}$ for $\epsilon$; note that
	$\displaystyle\frac{c-2}{2c}n_N$ is at most half of the gap between $\tilde{x}(N,-)$ and $x(N,-)$ thanks to Lemma~\ref{lem:introducing x(N,+)...}. 
	
	\begin{proposition}\label{prop:small d(N,i) f(N,j)}
	Assume that $P$ satisfies:
	\begin{equation}\label{eq:P condition}
	\frac{2^{P-1}+(P-1)2^{P-2}}{c^{P-1}}<\frac{c-2}{4c}.
	\end{equation}
	Then for all but finitely many $N\in\scrN$,
	we have $\displaystyle d(N,i)\leq \tilde{x}(N,-)+\frac{c-2}{2c} n_N$
	and $f(N,j)\leq \displaystyle \tilde{x}(N,-)+\frac{c-2}{2c} n_N$ for
	$1\leq i\leq D$ and $1\leq j\leq F$.
	\end{proposition}
	\begin{proof}
	We prove by contradiction and without loss of generality, we only
	need to consider 2  cases:
	\begin{itemize}
		\item Case 1: there exists an infinite subset $\scrN'$ of $\scrN$ such that $d(N,1)>\tilde{x}(N,-)+\frac{c-2}{2c}n_N$ for every
		$N\in \scrN'$.
		
		\item Case 2: there exists an infinite subset $\scrN'$ of $\scrN$ such that $f(N,1)>\tilde{x}(N,-)+\frac{c-2}{2c}n_N$ for every
		$N\in \scrN'$.
	\end{itemize}
	
	First, we assume Case 1. Let $\epsilon>0$ be a small number that 
	will be specified later. For all sufficiently large $N\in\scrN'$, 
	we have:
	\begin{align}	
	\begin{split}	
		\vert\sum_{i=1}^D(s-s_{N-P})d_{N,i}\alpha^{d(N,i)}-\sum_{j=1}^F f_{N,j}\alpha^{f(N,j)}\vert
		&=\vert (s-s_{N-1})x_N\vert\\
		&<\vert\alpha\vert^{\tilde{x}(N,-)+\epsilon n_N}
	\end{split}
	\end{align}
	by \eqref{eq:exist DF last property} and Lemma~\ref{lem:introducing x(N,+)...}.
	We now apply the Subspace Theorem over the field $\Q(\alpha)$. Let $S=\{v_{\infty},w\}$
	be its archimedean places as described in the proof of Proposition~\ref{prop:s in Qalpha}. We work with linear forms in the variables:
	$$(X_1,\ldots,X_D,Y_1\ldots,Y_F)$$
	 and the vectors
	 \begin{align*}
	 \bfv_N=(((s-s_{N-P})d_{N,i}\alpha^{d(N,i)})_{1\leq i\leq D},(-f_{N,j}\alpha^{f(N,j)})_{1\leq j\leq F})
	 \end{align*}
	for large $N\in\scrN'$. For $v\in S$, the linear forms are denoted
	 $L_{v,X,i}$ for $1\leq i\leq D$, and $L_{v,Y,j}$ for
	 $1\leq j\leq F$. They are defined as follows:
	 \begin{itemize}
		\item For any $v\in S$, $L_{v,X,i}=X_i$ for $2\leq i\leq D$
		and $L_{v,Y,j}=Y_j$ for $1\leq j\leq F$.
		\item If $v=w$, define $L_{v,X,1}=X_1$. 
		\item If $v=v_\infty$, define
		$$L_{v,X,1}=X_1+\ldots+X_D+Y_1+\ldots+Y_F.$$ 
	 \end{itemize}
	Since $s$ is irrational by Corollary~\ref{cor:s is irrational}
	and $\vert s-s_{N-P}\vert=o(1)$, we have
	$\vert s-s_{N-P}\vert_w=O(1)$. Together with the fact that
	the $d_{N,i}$'s and $f_{N,j}$'s have multiplicative height
	$e^{o(n_N)}$, for all large $N\in\scrN'$, we have:
	\begin{equation}\label{eq:DF ST 1}
	\prod_{v\in S}\vert L_{v,X,1}(\bfv_N)\vert_v<\vert\alpha\vert^{(-d(N,1)+\tilde{x}(N,-)+\epsilon n_N)/2}<\vert\alpha\vert^{(-(c-2)/(2c)+\epsilon)n_N/2}
	\end{equation}
	where the last inequality follows from the assumption
	in Case 1.

	As in the proof of Proposition~\ref{prop:s in Qalpha} and using $H(s-s_{N-P})=O(H(s_{N-P}))$, we have:
	\begin{align}\label{eq:DF ST 2}
	\begin{split}
	H(\tilde{\bfv}_N)^{D+F}\prod_{v\in S}\prod_{L}\frac{1}{\vert\bfv_N\vert_v}&=\left(\prod_{v\in M_K\setminus S}\vert\bfv_N\vert_v\right)^{D+F}\\
	&\leq H(s_{N-P})^{D+F}\vert\alpha\vert^{\epsilon n_N}\\
	&\leq \vert\alpha\vert^{(D+F)n_{N-P+1}+\epsilon n_N}
	\end{split}
 	\end{align}
 	for all sufficiently large $N\in\scrN'$. Recall that $D+F\leq 2^{P-1}+(P-1)2^{P-2}$, as before,
 	we can apply the Subspace Theorem if:
 	$$\frac{2^{P-1}+(P-1)2^{P-2}}{c^{P-1}}+\epsilon +\frac{1}{2}\left(-\frac{c-2}{2c}+\epsilon\right)<0$$
 	or in other words
 	$$\frac{2^{P-1}+(P-1)2^{P-2}}{c^{P-1}}+\frac{3}{2}\epsilon<\frac{c-2}{4c}.$$
 	We can choose such an $\epsilon$ thanks to the given condition on $P$. 
 	Then the Subspace Theorem yields a non-trivial linear relation
 	over $\Q(\alpha)$ among the $(s-s_{N-P})d_{N,i}\alpha^{d(N,i)}$'s and $f_{N,j}\alpha^{f(N,j)}$'s for $N$ in an infinite subset 
 	$\scrN''$ of $\scrN'$. This allows us to express one of them as a linear combination of the other terms and we obtain a new data satisfying the properties stated in Proposition~\ref{prop:exist D and F} in which $\scrN$ is replaced by $\scrN''$ and $D+F$ is replaced by $D+F-1$, contradicting the minimality of $D+F$. This shows that Case 1 cannot happen.
 	
 	By similar arguments, we have that Case 2 cannot happen either. For Case 2, we consider the same $\bfv_N$'s, the variables
 	$X_1,\ldots,X_D,Y_1,\ldots,Y_F$, and $S$ as in Case 1 while the linear forms $L_{x,X,i}$'s and $L_{v,Y,j}$'s are defined as folows:
 	\begin{itemize}
 		\item For any $v\in S$, $L_{v,X,i}=X_i$ for $1\leq i\leq D$
 		and $L_{v,Y,j}=Y_j$ for $2\leq j\leq F$.
 		
 		\item If $v=w$, define $L_{v,Y,1}=Y_1$.
 		
 		\item If $v=v_{\infty}$, define $L_{v,Y,1}=X_1+\ldots+X_D+Y_1+\ldots+Y_F$.
 	\end{itemize}
 	Then we proceed as before and arrive at a contradiction.
 	\end{proof}	
	
	We are now at the final stage of the proof of Theorem~\ref{thm:general}.
	\begin{notation}\label{conv:2^(P-1) for x(N,-)} 
	Let 
	$i^-\in\{1,\ldots,2^{P-1}\}$ such that that $x(N,i^-)=x(N,-)$.
	Let
	$$I^-:=\{i\in\{1,\ldots,2^{P-1}\}:\ x(N,i)<x(N,i^-)\} \text{and}$$
	$$I^+:=\{i\in\{1,\ldots,2^{P-1}\}:\ x(N,i)>x(N,i^-)\}.$$
	Note that $\vert I^-\vert=2^{P-2}-1$ and $\vert I^+\vert=2^{P-2}$. 
	\end{notation}
	
	We apply $\sigma$ to both sides of:
	\begin{equation}
	\sum_{i=1}^D(s-s_{N-P})d_{N,i}\alpha^{d(N,i)}-\sum_{j=1}^F f_{N,j}\alpha^{f(N,j)}
		=(s-s_{N-1})x_N,
	\end{equation}	
	and recall that $\sigma(\alpha)=\beta=\displaystyle\pm\frac{1}{\alpha}$ to get:
	\begin{align}
	\begin{split}
	&\sum_{i=1}^D(\sigma(s)-s_{N-P})\sigma(d_{N,i})(\pm 1)^{d(N,i)}\alpha^{-d(N,i)}-\sum_{j=1}^F \sigma(f_{N,j})(\pm 1)^{f(N,j)}\alpha^{-f(N,j)}\\
		&=(\sigma(s)-s_{N-1})x_N.
	\end{split}
	\end{align}
	Then taking the difference of the previous 2 equations, we get:
	\begin{align}\label{eq:difference of 2 equations}
	\begin{split}
	(\sigma(s)-s)x_N=&\sum_{i=1}^D(\sigma(s)-s_{N-P})\sigma(d_{N,i})(\pm 1)^{d(N,i)}\alpha^{-d(N,i)}\\
	&-\sum_{j=1}^F \sigma(f_{N,j})(\pm 1)^{f(N,j)}\alpha^{-f(N,j)}\\
	&-\sum_{i=1}^D(s-s_{N-P})d_{N,i}\alpha^{d(N,i)}+\sum_{j=1}^F f_{N,j}\alpha^{f(N,j)}.
	\end{split}
	\end{align}
	
	Our idea to conclude the proof of Theorem~\ref{thm:general} is as 
	follows. Note that $\sigma(s)-s\neq 0$, hence the LHS of \eqref{eq:difference of 2 equations} contains the term
	$(\sigma(s)-s)x_{N,i^-}\alpha^{x(N,-)}$. Let $\delta_2=\displaystyle\frac{2(c-2)}{5^{P-1}(c-1)}$ as in the proof of 
	Lemma~\ref{lem:expression for xN and x'N} (for $Q=P-1$) and we have
	that there is a gap $\delta_2n_N$ between any two of the $x(N,i)$'s. However, since $\delta_2$ depends on $P$, the previous method of choosing a sufficiently large $P$ does not work if we ``play the $x(N,i)$'s against each other''. The whole point of 
	Lemma~\ref{lem:introducing x(N,+)...} and 
	Proposition~\ref{prop:small d(N,i) f(N,j)} is that we now
	have a gap of at least 
	$\displaystyle\frac{c-2}{2c}n_N$ between $x(N,-)$ and 
	any of the $\pm d(N,i)$'s,
	$\pm f(N,j)$'s
	and a gap of at least 
	$\displaystyle\frac{2(c-2)}{5(c-1)}n_N$ between $x(N,-)$ and any of the $x(N,i)$ with $i\in I^{+}$.
	
	\begin{proposition}\label{prop:exist A,B,C,V,W}
	Recall $\delta_2=\displaystyle\frac{2(c-2)}{5^{P-1}(c-1)}$ and let $\theta=\min\left\{\displaystyle\frac{c-2}{2c},\displaystyle\frac{2(c-2)}{5(c-1)}\right\}$. There exist integers $A,B,C,V,W\geq 0$, an infinite subset $\scrN'$ of $\scrN$, and tuples 
	\begin{itemize}
	\item $(a_{N,1},\ldots,a_{N,A})$, $(a(N,1),\ldots,a(N,A))$, 
	\item $(b_{N,1},\ldots,b_{N,B})$, $(b(N,1),\ldots,b(N,B))$,
	\item $(c_{N,1},\ldots,c_{N,C})$, $(c(N,1),\ldots,c(N,C))$,
	\item $(v_{N,1},\ldots,v_{N,V})$, $(v(N,1),\ldots,v(N,V))$, 
	\item $(w_{N,1},\ldots,w_{N,W})$ and $(w(N,1),\ldots,w(N,W))$
	\end{itemize}
	for each $N\in\scrN'$ with the following properties:
	\begin{itemize}
		\item [(i)] $A+B+C+V+W\leq 2^{P-1}-1+2(D+F)\leq 2^P+P2^{P-1}-1$.
		\item [(ii)] For every $N\in\scrN'$, the $a(N,i)$'s, $b(N,j)$'s, $c(N,k)$, $v(N,\ell)$'s, and $w(N,m)$'s are integers
		for every $i,j,k,\ell,m$. 
		\item [(iii)] For every $N\in\scrN'$, we have:
		$$\max_{i,j,k,\ell,m}\{\vert a(N,i)\vert,\vert b(N,j)\vert, \vert c(N,k)\vert,\vert v(N,\ell)\vert, \vert w(N,m)\vert\}<2n_N.$$
		\item [(iv)] For every $N\in\scrN'$, we have $a(N,i)\leq -\delta_2 n_N$, $b(N,j)\leq -\theta n_N$, $c(N,k)\geq \theta n_N$, $v(N,\ell)\leq-\theta n_N$, and $w(N,m)\geq \theta n_N$ for every $i,j,k,\ell,m$.
		\item [(v)] For every $N\in\scrN'$, the $a_{N,i}$'s, $b_{N,j}$'s, $c_{N,k}$'s, $v_{N,\ell}$'s, and $w_{N,m}$'s are elements of $\Q(\alpha)$.
		\item [(vi)] As $N\to\infty$, we have
		$h(a_{N,i})/n_N\to 0$, $h(b_{N,j})/n_N\to 0$, $h(c_{N,k})/n_N\to 0$, $h(v_{N,\ell})/n_N\to 0$, and $h(w_{N,m})/n_N\to 0$ for every $i,j,k,\ell,m$. 
		\item [(vii)] For every $N\in\scrN'$, we have:
		\begin{align}\label{eq:exist ABCVW, last property}
		\begin{split}
			\sigma(s)-s=&\sum_{i=1}^A a_{N,i}\alpha^{a(N,i)}+\sum_{j=1}^B b_{N,j}\alpha^{b(N,j)}+\sum_{k=1}^Cc_{N,k}\alpha^{c(N,k)}\\
			&\sum_{\ell=1}^V (s-s_{N-P})v_{N,\ell}\alpha^{v(N,\ell)}+\sum_{m=1}^W (\sigma(s)-s_{N-P})w_{N,m}\alpha^{w(N,m)}.
		\end{split}
		\end{align}
	\end{itemize}
	\end{proposition}
	\begin{proof}
	We use \eqref{eq:difference of 2 equations} and the expression for $x_N$ in Lemma~\ref{lem:expression for xN and x'N}, then divide both sides by
	$x_{N,i^-}\alpha^{x(N,-)}$ to get:
	\begin{align}
	\begin{split}
	 \sigma(s)-s=&-\sum_{i\in I^-}\frac{(\sigma(s)-s)x_{N,i}}{x_{N,i^-}}\alpha^{x(N,i)-x(N,-)}\\
	 &-\sum_{i\in I^+}\frac{(\sigma(s)-s)x_{N,i}}{x_{N,i^-}}\alpha^{x(N,i)-x(N,i^-)}\\
	 &-\sum_{i=1}^F\frac{\sigma(f_{N,i})}{x_{N,i^-}}(\pm 1)^{f(N,i)}\alpha^{-f(N,i)-x(N,-)}\\
	 &+\sum_{i=1}^F\frac{f_{N,i}}{x_{N,i^-}}\alpha^{f(N,i)-x_(N,-)}\\
	&-\sum_{i=1}^D(s-s_{N-P})\frac{d_{N,i}}{x_{N,i^-}}\alpha^{d(N,i)-x(N,-)}\\
	&+\sum_{i=1}^D(\sigma(s)-s_{N-P})\frac{\sigma(d_{N,i})}{x_{N,i^-}}(\pm 1)^{d(N,i)}\alpha^{-d(N,i)-x(N,-)}
	\end{split}
	\end{align}
	Let $\scrN'$ be the set of all sufficiently large $N\in\scrN$; in the following, $N$ is an element of $\scrN'$.
	We want $\displaystyle\sum_{i=1}^A a_{N,i}\alpha^{a(N,i)}$ to be
	$\displaystyle-\sum_{i\in I^-}\frac{(\sigma(s)-s)x_{N,i}}{x_{N,i^-}}\alpha^{x(N,i)-x(N,-)}$.
	Therefore we let $A=\vert I^-\vert=2^{P-2}-1$,
	let the $a_{N,i}$ and $a(N,i)$ for $1\leq i\leq A$ be 
	respectively
	the $-\displaystyle\frac{(\sigma(s)-s)x_{N,i}}{x_{N,i^-}}$ and
	$x(N,i)-x(N,-)$ for $i\in I^-$. By 
	Lemma~\ref{lem:expression for xN and x'N}
	and the definition of $I^-$, we have 
	$a(N,i)\leq -\delta_2n_N$.
	 
	We want $\displaystyle\sum_{j=1}^B b_{N,j}\alpha^{b(N,j)}$ to
	be $\displaystyle\sum_{i=1}^F\frac{f_{N,i}}{x_{N,i^-}}\alpha^{f(N,i)-x_(N,-)}$. Therefore we let $B=F$, let the $b_{N,j}$ and $b(N,j)$ for $1\leq j\leq B$ be respectively the
	$\displaystyle\frac{f_{N,i}}{x_{N,i^-}}$ and 
	$f(N,i)-x(N,-)$ for $1\leq i\leq F$. By 
	Lemma~\ref{lem:introducing x(N,+)...} and 
	Proposition~\ref{prop:small d(N,i) f(N,j)}, we have:
	$$f(N,i)-x(N,-)\leq \tilde{x}(N,-)+\frac{c-2}{2c}n_N-x(N,-)\leq -\frac{c-2}{2c}n_N\leq -\theta n_N$$
	for $1\leq i\leq F$.
	
	We want $\displaystyle\sum_{k=1}^C c_{N,k}\alpha^{c(N,k)}$ to
	be:
	$$-\sum_{i\in I^+}\frac{(\sigma(s)-s)x_{N,i}}{x_{N,i^-}}\alpha^{x(N,i)-x(N,i^-)}-\sum_{i=1}^F\frac{\sigma(f_{N,i})}{x_{N,i^-}}(\pm 1)^{f(N,i)}\alpha^{-f(N,i)-x(N,-)}.$$
	We let $C=\vert I^+\vert+F=2^{P-2}+F$ and specify the $c_{N,k}$ and $c(N,k)$
	in the same manner as before. Note that $x(N,+)$ is
	the minimum among the $x(N,i)$ for $i\in I^+$ while
	Lemma~\ref{lem:introducing x(N,+)...} and 
	Proposition~\ref{prop:small d(N,i) f(N,j)} yields
	$$-f(N,i)\geq -\tilde{x}(N,-)-\frac{c-2}{2c}n_N=\tilde{x}(N,+)-\frac{c-2}{2c}n_N>x(N,+).$$
	This guarantees $c(N,k)\geq x(N,+)-x(N,-)\geq \displaystyle\frac{2(c-2)}{5(c-1)}n_N\geq \theta n_N$.
	
	Finally, we want $\displaystyle\sum_{\ell=1}^V (s-s_{N-P})v_{N,\ell}\alpha^{v(N,\ell)}$ to be 
	$$\displaystyle-\sum_{i=1}^D(s-s_{N-P})\frac{d_{N,i}}{x_{N,i^-}}\alpha^{d(N,i)-x(N,-)}$$
	and want $\displaystyle\sum_{m=1}^W (\sigma(s)-s_{N-P})w_{N,m}\alpha^{w(N,m)}$ to be
	$$\sum_{i=1}^D(\sigma(s)-s_{N-P})\frac{\sigma(d_{N,i})}{x_{N,i^-}}(\pm 1)^{d(N,i)}\alpha^{-d(N,i)-x(N,-)}.$$
	We let $V=W=D$ and similar arguments can be used to finish the proof; note that with our choice:
	$$A+B+C+V+W=2^{P-1}-1+2(D+F)\leq 2^P+P2^{P-1}-1$$
	where the last inequality follows from 
	Proposition~\ref{prop:exist D and F}.
	\end{proof}
	
	\begin{definition}\label{def:minimal A+B+C+V+W}
	Among all the collections of data $(A,B,C,V,W,\scrN',\ldots)$ satisfying properties (i)-(vii) in Proposition~\ref{prop:exist A,B,C,V,W}, we choose one
	for which $A+B+C+V+W$ is minimal. By abusing the notation, we
	still use the same notation $A$, $B$, $C$, $V$, $W$, $\scrN'$, $a_{N,i}$'s, $a(N,i)$'s, $b_{N,j}$'s, $b(N,j)$'s, $c_{N,k}$'s,
	$c(N,k)$'s, $v_{N,\ell}$'s, $v(N,\ell)$'s, $w_{N,m}$, and
	$w(N,m)$'s for
	this chosen data with minimal $A+B+C+V+W$.
	\end{definition}
	
	Another application of the Subspace Theorem yields the following:
	\begin{proposition}\label{prop:B=C=V=W=0}
	Recall that $\theta=\displaystyle\min\left\{\frac{c-2}{2c},\frac{2(c-2)}{5(c-1)}\right\}$. Assume that $P$ satisfies:
	\begin{equation}\label{eq:P final condition}
	\frac{2^P+P2^{P-1}-1}{c^{P-1}}<\frac{\theta}{2}.
	\end{equation}
	Then we have $B=C=V=W=0$.
	\end{proposition}
	\begin{proof}
	First, suppose that $B>0$. Let $\epsilon>0$ be a small number that will be specified later. We apply the Subspace Theorem over the field
	$\Q(\alpha)$ and let $S=\{v_{\infty},w\}$ be as before. We work
	with linear forms in the variables:
	$$((X_i)_{1\leq i\leq A},(Y_j)_{1\leq j\leq B},(Z_k)_{1\leq k
	\leq C}, (R_\ell)_{1\leq \ell\leq V}, (T_m)_{1\leq m\leq W})$$
	and the vectors
	\begin{align*}
		\bfv_N=(&(a_{N,i}\alpha^{a(N,i)})_{1\leq i\leq A},(b_{N,j}\alpha^{b(N,j)})_{1\leq j\leq B},((c_{N,k}\alpha^{c(N,k)})_{1\leq k\leq C},\\
		&((s-s_{N-P})v_{N,\ell}\alpha^{v(N,\ell)})_{1\leq\ell\leq V},((\sigma(s)-s_{N-P})w_{N,m}\alpha^{w(N,m)})_{1\leq m\leq W})
	\end{align*}
	for $N\in\scrN'$.
	For $v\in S$, the linear forms are denoted $L_{v,X,i}$,
	$L_{v,Y,j}$, $L_{v,Z,k}$, $L_{v,R,\ell}$, and $L_{v,T,m}$
	for $1\leq i\leq A$, $1\leq j\leq B$, $1\leq k\leq C$,
	$1\leq \ell\leq V$ and $1\leq m\leq W$
	and they are defined as follows:
	\begin{itemize}
	\item For any $v\in S$, $L_{v,X,i}=X_i$, $L_{v,Y,j}=Y_j$,
	$L_{v,Z,k}=Z_k$, $L_{v,R,\ell}=R_{\ell}$, and $L_{v,T,m}=T_m$
	for every $i,j,k,\ell,m$ \emph{except} when $j=1$.
	
	\item If $v=v_\infty$, define $L_{v,Y,1}=Y_1$.
	
	\item If $v=w$, define
	$$L_{v,Y,1}=\sum_{i=1}^A X_i+\sum_{j=1}^B Y_j+\sum_{k=1}^C Z_k+\sum_{\ell=1}^V R_\ell+\sum_{m=1}^W T_m.$$
	\end{itemize}
	Therefore if $v=v_\infty$, we have $\vert L_{v,Y,1}(\bfv_N)\vert_v=\vert b_{N,1}\vert^{1/2}\vert\alpha\vert^{b(N,1)/2}\leq \vert\alpha\vert^{(-\theta n_N/2)+o(n_N)}$ since $b_{N,1}$ has height $o(n_N)$ 
	while $b(N,1)\leq -\theta n_N$. If $v=w$, we have
	$L_{v,Y,1}(\bfv_N)=\sigma(s)-s$ thanks to
	\eqref{eq:exist ABCVW, last property}. Thus, arguing as before,
	 we have:
	 \begin{align}
	 \prod_{v\in S}\prod_{L}\vert L(\bfv_N)\vert_v< \vert\alpha\vert^{(-\theta+\epsilon)n_N/2}
	 \end{align}
	for all sufficiently large $N\in\scrN'$ where $L$ ranges over 
	all the $L_{v,X,i}$'s, $L_{v,Y,j}$'s, $L_{v,Z,k}$'s, $L_{v,R,\ell}$'s, and $L_{v,T,m}$. On the other hand,
	\begin{align}
	\begin{split}
		H(\tilde{\bfv}_N)^{A+B+C+V+W}\prod_{v\in S}\prod_{L}\frac{1}{\vert\bfv_N\vert_v}&=\left(\prod_{v\in M_K\setminus S}\vert\bfv_N\vert_v\right)^{A+B+C+V+W}\\
		&\leq H(s_{N-P})^{A+B+C+V+W}\alpha^{\epsilon n_N}\\
		&\leq \vert\alpha\vert^{(A+B+C+V+W)n_{N-P+1}+\epsilon n_N}.
	\end{split}
	\end{align}
	Since $A+B+C+V+W\leq 2^P+P2^{P-1}-1$, we can apply the Subspace Theorem if:
	$$\frac{2^P+P2^{P-1}-1}{c^{P-1}}+\epsilon+\frac{-\theta+\epsilon}{2}<0.$$
	At the beginning of the proof, can choose an $\epsilon$ satisfying the above inequality thanks to the condition on $P$. Then the Subspace Theorem implies that the coordinates of $\bfv_N$ satisfies a 
	non-trivial linear relation over $\Q(\alpha)$ for every $N$ in an infinite subset $\scrN''$ of $\scrN'$. Then we have a new data 
	satisfying the properties in Proposition~\ref{prop:exist A,B,C,V,W} in which
	$\scrN'$ is replaced by $\scrN''$ and $A+B+C+V+W$ is replaced
	by $A+B+C+V+W-1$; this contradicts the minimality of $A+B+C+V+W$.
	
	For the case $C>0$, $V>0$, or $W>0$, we use the same vectors $\bfv_N$ and the same notation for the variables and linear forms. In the case $C>0$, the linear forms are:
	\begin{itemize}
	\item For any $v\in S$, $L_{v,X,i}=X_i$, $L_{v,Y,j}=Y_j$,
	$L_{v,Z,k}=Z_k$, $L_{v,R,\ell}=R_{\ell}$, and $L_{v,T,m}=T_m$
	for every $i,j,k,\ell,m$ \emph{except} when $k=1$.
	
	\item If $v=w$, define $L_{v,Z,1}=Z_1$.
	
	\item If $v=v_\infty$, define
	$$L_{v,Z,1}=\sum_{i=1}^A X_i+\sum_{j=1}^B Y_j+\sum_{k=1}^C Z_k+\sum_{\ell=1}^V R_\ell+\sum_{m=1}^W T_m.$$
	\end{itemize}
	
	In the case $V>0$, the linear forms are:
	\begin{itemize}
	\item For any $v\in S$, $L_{v,X,i}=X_i$, $L_{v,Y,j}=Y_j$,
	$L_{v,Z,k}=Z_k$, $L_{v,R,\ell}=R_{\ell}$, and $L_{v,T,m}=T_m$
	for every $i,j,k,\ell,m$ \emph{except} when $\ell=1$.
	
	\item If $v=v_\infty$, define $L_{v,R,1}=R_1$.
	
	\item If $v=w$, define
	$$L_{v,R,1}=\sum_{i=1}^A X_i+\sum_{j=1}^B Y_j+\sum_{k=1}^C Z_k+\sum_{\ell=1}^V R_\ell+\sum_{m=1}^W T_m.$$
	\end{itemize}
	Finally, in the case $W>0$, the linear forms are:
	\begin{itemize}
	\item For any $v\in S$, $L_{v,X,i}=X_i$, $L_{v,Y,j}=Y_j$,
	$L_{v,Z,k}=Z_k$, $L_{v,R,\ell}=R_{\ell}$, and $L_{v,T,m}=T_m$
	for every $i,j,k,\ell,m$ \emph{except} when $m=1$.
	
	\item If $v=w$, define $L_{v,T,1}=T_1$.
	
	\item If $v=v_\infty$, define
	$$L_{v,T,1}=\sum_{i=1}^A X_i+\sum_{j=1}^B Y_j+\sum_{k=1}^C Z_k+\sum_{\ell=1}^V R_\ell+\sum_{m=1}^W T_m.$$
	\end{itemize}
	
	Then similar arguments as before lead to a contradiction. This finishes the proof.
	\end{proof}
	
	\begin{proof}[Completion of the proof of Theorem~\ref{thm:general}]
	At the beginning of this section, we fix a sufficiently large $P$
	satisfying both \eqref{eq:P condition} and \eqref{eq:P final condition}. Then the previous results show that there exist an infinite 
	set of positive integers $\scrN'$, an integer $A\geq 0$, tuples $(a_{N,1},\ldots,a_{N,A})$
	and $(a(N,1),\ldots,a(N,A))$ satisfying the conditions of
	Proposition~\ref{prop:exist A,B,C,V,W}; in particular:
	$$\sigma(s)-s=a_{N,1}\alpha^{a(N,1)}+\ldots+a_{N,A}\alpha^{a(N,A)}$$
	for every $N\in\scrN'$. However, each $\vert a_{N,i}\vert=\vert\alpha\vert^{o(n_N)}$
	as $N\to\infty$ while each $a(N,i)<\displaystyle-\frac{2(c-2)}{5^{P-1}(c-1)}n_N$. Let $N\to\infty$ the we have
	$$\sigma(s)-s=0$$
	contradicting the earlier results that $s\in\Q(\alpha)$ is irrational. This finishes the proof.
	\end{proof}
	
	\bibliographystyle{amsalpha}
	\bibliography{EG} 	

\end{document}